\documentclass[10pt,a4paper]{amsart}
\RequirePackage{amsthm,amsmath,amsfonts,amssymb,enumitem}
\RequirePackage[authoryear]{natbib}
\RequirePackage[colorlinks,citecolor=blue,urlcolor=blue]{hyperref}
\RequirePackage{graphicx}
\usepackage[margin=1.35in]{geometry}

\newtheorem{theorem}{Theorem}
\newtheorem{lemma}{Lemma}
\newtheorem{proposition}{Proposition}
\newtheorem{corollary}{Corollary}
\theoremstyle{remark}

\newtheorem{example}{Example}

\newtheorem{remark}{Remark}
\def\d{\mathop{\rm d\!}\nolimits}
\def\dd{\mathop{\rm d}\nolimits}
\def\vol{\mathop{\rm vol}\nolimits}

\def\M{\mathop{\boldsymbol M}\nolimits}

\def\TM{\mathop{\boldsymbol{T\!M}}\nolimits}

\newcommand{\E}[1]{\operatorname{E}\left[#1\right]}
\newcommand{\red}{\textcolor{red}}
\begin{document}
\title{A diffusion approach to Stein's method \\ on Riemannian manifolds}
%\author{H. Le\thanks{Huiling.Le@nottingham.ac.uk}, A. Lewis, K. Bharath and C. Fallaize}
\author{Huiling Le}% \thanks{Huiling.Le@nottingham.ac.uk}
\email{Huiling.Le@nottingham.ac.uk}
\author{Alexander Lewis}
\author {Karthik Bharath}
\email{Karthik.Bharath@nottingham.ac.uk}
\author{Christopher Fallaize}
\address{School of Mathematical Sciences, 
University of Nottingham, UK}
\date{}
%\subjclass[2000]{60D05, 60J60, 60H99, 58C05}
\maketitle

	\begin{abstract}
				We detail an approach to developing Stein's method for bounding integral metrics on probability measures defined on a Riemannian manifold $\M$.  Our approach exploits the relationship between the generator of a diffusion on $\M$ having a target invariant measure and its characterising Stein operator. We consider a pair of such diffusions with different starting points, and through analysis of the distance process between the pair, derive Stein factors, which bound the solution to the Stein equation and its derivatives. The Stein factors contain curvature-dependent terms and reduce to those currently available for $\mathbb R^m$, and moreover imply that the bounds for $\mathbb R^m$ remain valid when $\M$ is a flat manifold. 
		\end{abstract}

\section{Introduction}
The eponymous method to estimate integral metrics and semi-metrics on spaces of probability measures proposed by Charles Stein \citep{CS} has led to tremendous improvements in distributional approximation techniques. See, for example, the surveys by \cite{BC} and \cite{ NR}. The method has mainly been developed for probability measures on $\mathbb{R}^m$ or $\mathbb{N}^m$ for $m \geq 1$. The focus of this paper is on developing a version of the method that can be employed to approximate probability measures on an $m$-dimensional Riemannian manifold. 

Abstracting Stein's method to a general space $\mathcal{X}$ in a heuristic manner is useful for elucidating its key ingredients and the ensuing challenges in developing a corresponding version on manifolds. The goal is to bound an integral (semi-)metric
\[
d_\mathcal{H}(\mu,\nu):=\sup_{h \in \mathcal{H}}\left|\int h \text{d} \mu -\int h \text{d}\nu\right|,
\]
between a probability measure $\nu$ and a target probability measure $\mu$ on $\mathcal{X}$ with respect to a class $\mathcal{H}$ of real-valued test functions on $\mathcal{X}$.  Stein's method is centred around the construction and study of an operator $L$ that maps functions  $f:\mathcal{X} \to \mathbb{R}$ in a certain class $\mathcal{F}$ into mean-zero functions under $\mu$: if $X \sim \mu$, then $\E{Lf(X)}=0$ for every $f \in \mathcal{F}$. The operator $L$ thus encodes information about $\mu$ and, when $\mathcal F$ is sufficiently large, one may determine a function $f_h \in \mathcal{F}$ associated with every $h \in \mathcal{H}$ that solves the \emph{Stein equation} (or the Poisson equation in PDE literature)
\[
h(x)-\E{h(X)}=Lf_h(x).
\]
As a consequence, bounding $d_\mathcal{H}(\mu,\nu)$ reduces to bounding the term $\sup_{f_h \in \mathcal{F}}\E{Lf_h(Z)}$, where $Z \sim \nu$, achieved in application-specific ways. An important implication, profitably used in some applications, is that the need to compute an expectation with respect to $\mu$ in $d_\mathcal{H}$ is circumvented; an example is when $\nu$ is the empirical measure based on points $x_1,\ldots,x_n$ on $\mathcal{X}$ and $\mu$ represents a conjectured limit probability measure. Upper bounds on $d_\mathcal{H}$ then depend explicitly on the smoothness of the functions in $\mathcal{F}$. Hence, integral to the success of Stein's method in upper bounding $d_\mathcal{H}$ are the following requirements: (a) construction of the operator $L$ and identifying its domain $\mathcal{F}$; and, (b) determination of the solution $f_h$ and its regularity properties. 

An introductory account on choices of $L$ satisfying requirement (a) for various probability measures $\mu$ on $\mathbb{R}$ (or some subset thereof) is available in \cite{NR}. When $\mathcal{X}=\mathbb{R}^m, m>1$, focus has mainly been restricted to the case when $\mu$ is a Gaussian measure (see \cite{CM, ADB, EM}) although more recently results on extensions to non-Gaussian measures have appeared in \cite{MG, MRS, FSX, CNXY}. 

An important observation by \cite{ADB} relates the operator $L$ to the infinitesimal generator of a diffusion process on $\mathbb{R}^m$ that solves an SDE with invariant measure $\mu$. This observation enables identification, and examination, of the solution to the Stein equation with the transition semigroup associated with $L$. The diffusion approach hence opens up the possibility of defining $L$ for $\mu$ on a manifold $\M$ by considering an SDE on $\M$ whose solution is a diffusion with invariant measure $\mu$. 

Broadly, this is the approach we adopt in this paper. On a complete Riemannian manifold $(\M,g)$ without boundary, we consider approximating probability measures of the form $\mu_\phi$ with density, up to a normalisation constant, $e^{-\phi}$ with respect to the volume measure $\d\vol$ for a smooth $\phi$. Under some conditions on $\phi$ and the geometry of $\M$, the diffusion with infinitesimal generator 
\[L_\phi:=\frac{1}{2}\left\{\Delta-\langle\nabla\phi,\,\nabla\rangle\right\}\]
has $\mu_\phi$ as its invariant measure, where $\nabla$ and $\Delta$ are the (Riemannian) gradient and Laplace-Bertrami operators, respectively. The operator $L_\phi$ generates mean-zero functions under $\mu_\phi$.

We address requirement (b) for generalising Stein's method to $\M$ by adapting the approach in \cite{MG} for log-concave measures $\mu$ on $\mathbb{R}^m$ to the manifold setting. In their paper, bounds on lower-order derivatives of the solution $f_h$, known as \emph{Stein factors} (see \cite{AR}), were derived by studying the distance between a pair of coupled diffusions $X_{t}$ and $Y_{t}$ with same invariant measure $\mu$ starting at distinct points. Analogously, we construct a pair of diffusions $X_{t,x}$ and $Y_{t,y}$ on $\M$ starting at $x$ and $y$ with identical generator $L_\phi$, and study the distance process $\rho(X_{t,x},Y_{t,y})$ around neighbourhoods of non-empty cut loci. In particular, when there is no first conjugate point contained in the cut locus to any given point in $\M$ we establish exponential \emph{pathwise} contraction for trajectories of the two diffusions towards their initial points; on the other hand when first conjugate points are present, we establish a similar contraction property that holds on \emph{average}. 

The study of the distance process enables the determination of Stein factors which bound the Lipschitz constants of the solution $f_h$, and its first and second derivatives, where the geometry of $\M$ manifests itself through curvature-dependent terms in the factors.  The derived bounds on $f_h$, as well as on its first and second derivatives, reduce to the ones of \cite{MG} for $\mathbb R^m$, which we show remain valid for complete, connected flat manifolds. The Stein factors are then used to construct upper bounds on integral (semi-)metrics between $\mu_\phi$ and another probability measure on $\M$ for specific choices of the class of test functions $\mathcal{H}$. In particular, using the first order bound on $f_h$, we derive an upper bound on the Wasserstein distance between $\mu_\phi$ and  $\mu_\psi$. A related generalisation of Stein's method to manifolds, based on the approach of \cite{FSX}, can be found in \cite{JT}. 

The paper is organised as follows. In Section \ref{notation} we define relevant quantities and introduce notation, and in Section \ref{assumptions} we describe assumptions on probability measures and diffusions under consideration, and the key condition \eqref{eqn2} and assumption (A1)  for the derivation of our results; the conditions are explicated with some examples. In Section \ref{distance} we describe the coupling of a pair of diffusions on $\M$ and analyse their distance process, when conjugate points are absent (Section \ref{noconjugate}) and present (Section \ref{conjugate}). In Section \ref{Stein_equation} we consider the Stein equation and its solution and derive Stein factor bounds that depend on the curvature of $\M$.  In Section \ref{bounds}, using the Stein factors, we derive bounds for integral (semi-)metrics: in Section \ref{wasserstein_bound} we derive an upper bound on the Wasserstein distance between $\mu_\phi$ and a probability measure of similar type, and in Section \ref{general_bound} we do the same for an integral semi-metric between $\mu_\phi$ and an arbitrary probability measure on $\M$.  
\section{Preliminaries}
\subsection{Notation and definitions}
\label{notation}
We assume throughout that  $(\M,g)$ is a complete and connected Riemannian manifold without boundary of dimension $m$ and with covariant derivative $D$; by $D^i, i>1$ we then denote higher orders of $D$. We shall denote by $\rho(x,y)$ the Riemannian distance between any two points $x$ and $y$ in $\M$, and by $\d\vol$ the Riemannian volume measure of $(\M,g)$.  We denote by $T_x(\M)$ the tangent space to $\M$ at $x\in\M$ and by $\TM$ the tangent bundle of $\M$. For $k \geq 1$, $\mathcal{C}^k(\M)$ denotes the class of $k$-times continuously differentiable real-valued functions on $\M$, $\mathcal{C}(\M)$ denotes the set of continuous functions, and $\mathcal{C}_0(\M)$ denotes continuous functions vanishing at infinity. The Lipschitz constant $C_0(h)$ of a Lipschitz function $h\in\mathcal C(\M)$ is defined as
\[C_0(h):=\sup_{x\not=y\in\M}\frac{|h(x)-h(y)|}{\rho(x,y)}\thinspace.\]
Higher-order Lipschitz constants of a function depend on bounding tensor fields. Accordingly, for each $x\in\M$ define the operator norm at $x$ for a tensor field $T$ on $\M$, based on $n$-fold tangent vectors at $x$, as
\[\|T\|_{op}:=\sup_{v_1,\cdots,v_n\in T_x(\M),\,|v_i|\not=0}\frac{|T(v_1,\cdots,v_n)|}{\prod\limits_{i=1}^n|v_i|}.\]
Then, if $h\in\mathcal C^k(\M)$, for $k \geq 1$, we may define 
\begin{equation}
\label{tensor_norm}	
C_i(h):=\sup_{\gamma_{x,y}, x\not=y\in\M}\frac{\|D^ih(x)-\Pi_{\gamma_{x,y}}(D^ih(y))\|_{op}}{\rho(x,y)},\qquad i=1,\ldots,k,
\end{equation}
and call them the Lipschitz constants of $D^ih$, where $\gamma_{x,y}$ denotes any possible minimal geodesic from $y$ to $x$ and $\Pi_{\gamma_{x,y}}$ denotes the parallel transport from $T_y(\M)$ to $T_x(\M)$ along $\gamma_{x,y}$. Note that $Dh=dh$ and that $\hbox{Hess}^h=D^2h$, where $\hbox{Hess}^h$ is the Hessian of $h$. Note also that $\sup_{x\in\M}\|D^{i+1}h(x)\|_{op}=C_i(h)$ for $i=0,\cdots,k-1$. Finally, if $X$ and $Z$ are two random variables on $\M$ with $X \sim \mu$ and $Z \sim \nu$, abusing notation, we interchangeably use $d_\mathcal{H}(\nu,\mu)$ and $d_\mathcal{H}(Z,X)$ to denote the integral (semi-)metric between the two probability measures, where 
$$d_\mathcal{H}(Z,X):=\sup_{h \in \mathcal H}|\E{ h(Z)}-\E{h(X)}|$$
with respect to a set of real-valued test functions $\mathcal H$. 

\subsection{Key assumptions}
\label{assumptions}
On $\M$, we consider probability measures of the form
\[
\d\mu_\phi=\frac{1}{c(\phi)}e^{-\phi}\d\vol,
\]
with $c(\phi)=\int_{\M}e^{-\phi}\d\vol$ $<\infty$ and with support on the entire space $\M$. We assume that $\phi\in\mathcal C^2(\M)$ is such that $\nabla\phi$ is Lipschitz; specifically, we assume that $D\phi$ has finite Lipschitz constant $C_1(\phi)$. Throughout, $X$ denotes a random variable with $X\sim\mu_\phi$.

The uniformly elliptic operator $L_\phi=1/2\left\{\Delta-\langle\nabla\phi,\,\nabla\rangle\right\}$ then is the infinitesimal generator of a Feller diffusion process that solves the It\^o stochastic differential equation
\begin{eqnarray}
	\d X_t=\d B^{\M}_t-\frac{1}{2}\nabla\phi(X_t)\,\d t,
	\label{eqn1}
\end{eqnarray}
where $B^{\M}_t$ is a Brownian motion on $\M$. 
If there is a constant $\kappa>0$ such that
\begin{eqnarray}
	\hbox{Ric}(x)+\hbox{Hess}^{\phi}(x)\geqslant-\kappa\,g(x),\quad\forall x\in\M,
	\label{eqn2}
\end{eqnarray}
where Ric is the Ricci curvature tensor, then the corresponding semigroup $P_t=e^{tL_\phi}$ is conservative (see \cite{DB}), i.e., $P_t1\equiv1$ for all $t>0$ or, equivalently, $X_t$ will, with probability one, not leave $\M$ in finite time. For a successful development of the Stein's method on $\M$, we need the \emph{Bakry-Emery curvature criterion}: there is a constant $\kappa>0$ such that, 
\begin{equation*}
(\text{A1}): \qquad \hbox{{\rm Ric}}(x)+\hbox{{\rm Hess}}^\phi (x)\geqslant2\kappa\,g(x) \quad\forall x\in\M.
	\label{eqn6a}
\end{equation*}
Evidently, assumption (A1) implies the condition in \eqref{eqn2}. 

\begin{remark}
	\label{remark1}
	When $\M=\mathbb R^m$, (A1) simplifies to $v^\top\hbox{Hess}^\phi v\geqslant2\kappa$ for any unit (column) vector $v$ in $\mathbb R^m$ where, as usual, $\hbox{Hess}^\phi$ is treated as an $m\times m$ matrix. Hence, (A1) reduces to the requirement in \cite{MG} that $-\phi$ is $2\kappa$-strongly concave, noting that in their notation, $\phi$ here is $-\log p$, up to a constant. This is also true if the Ricci curvature of $\M$ is always non-positive. In general, (A1) is weaker than the requirement that $-\phi$ is $c$-strongly concave for some $c>0$. 
\end{remark}

\begin{example}
	\label{examples}
In order to elucidate condition \eqref{eqn2} and assumption (A1) we look at a few example manifolds and probability measures $\mu_\phi$. 
\begin{enumerate}[label=(\roman*), leftmargin=*]
	\itemsep 1.3em
	\item  \emph{$\M$ is the standard sphere $S^m$ of dimension $m$}.  The function $\phi(x)$ corresponding to the von Mises-Fisher distribution $M_m(x_0,c)$ takes the form $\phi(x)=-c\cos(r(x))$, with $r(x)=\rho(x_0,x)$ for $c>0$ and a fixed point $x_0 \in \M$. Since $D^2f(r)=f''(r)\d r\times\d r+f'(r)D^2 r$ on general manifolds and since $D^2r=\cot(r)\{g-\d r\times\d r\}$ on $S^m$ (see \cite{GW}), it follows that
	\[\hbox{Hess}^\phi(x)=-c\,D^2\cos(r(x))=c\,\cos(r(x))\,g(x);\]
this ensures that 
	\[\hbox{Ric}(x)+\hbox{Hess}^\phi(x)\geqslant\{(m-1)-c\}g(x),\]
and condition \eqref{eqn2} holds for the von Mises-Fisher distribution with $\kappa>\max\{-(m-1)+c,0\}$.  However, if there is a $\kappa>0$ such that assumption (A1) holds, then we must have $0<c<m-1$. This requires in particular $m>1$, and thus any von Mises-Fisher distribution on the circle fails to satisfy (A1).
	\item 
	\emph{$\M$ is hyperbolic space  $\mathbb H^m$ with sectional curvature $-1$}.  For $\phi(x)=c\rho(o,x)^2$ where $c>0$ and $o$ is a fixed point in $\M$, we have $\int_{\M}e^{-\phi}\d\hbox{vol}<\infty$ as, in terms of normal coordinates at $o$, $\d\hbox{vol}=\sinh(\rho)^{m-1}\d\rho\d\theta$. On the other hand, Hess$^\phi(x)\geqslant 2c g(x)$ by the Hessian Comparison Theorem and Ric$(x)=-(m-1)g(x)$. Hence condition \eqref{eqn2} holds with $\kappa>\max\{(m-1)-2c,0\}$ for such a $\phi$. Moreover, if $c>(m-1)/2$, then there is a $\kappa>0$ such that assumption (A1) holds.
	%%%%%%%%%
	\item
	\emph{$\M$ is the complex projective space $\mathbb{CP}^m$ equipped with the Fubini-Study metric. This is also the Kendall shape space of configurations in $\mathbb R^2$ with $m+1$ labelled landmarks.} Let $A$ be an $(m+1)\times(m+1)$ Hermitian matrix, i.e. $A=A^*$ and $\phi(z)=-z^*Az$, for $z=x+iy\in\mathbb C^{m+1}$ (column vectors) and $|z|=1$, where $A^*$ denotes the complex conjugate transpose of $A$. Without loss of generality, we may assume that the smallest eigenvalue of $A$ is zero. The corresponding $\mu_\phi$ is the complex Bingham distribution on $\mathbb CS^m=S^{2m+1}$. Since $\phi(z)=\phi(e^{i\theta}z)$, $\mu_\phi$ can be regarded as a distribution on $\M$ (see \cite{JTK}). It can be shown that Hess$^\phi(w,w)=2\{\phi(z)-\phi(w)\}\geqslant-2\lambda_{\max}$ for a horizontal (with respect to the projection from $S^{2m+1}$ to $\mathbb{CP}^m$) unit vector $w\in T_z(S^{2m+1})$, where $\lambda_{\max}>0$ is the largest eigenvalue of $A$.
	
The complex projective space $\mathbb{CP}^m$ equipped with the Fubini-Study metric is an Einstein manifold with its Ricci curvature tensor equal to $2(m+1)$ times the metric tensor. Thus,
\[\text{Ric}+\text{Hess}^\phi\geqslant2\left\{m+1-\lambda_{\max}\right\}g,\]
and so, for the complex Bingham distribution on $\mathbb{CP}^m$, condition \eqref{eqn2} holds with $\kappa>2\max\{\lambda_{\max}-(m+1),0\}$ and assumption (A1) holds if $\lambda_{\max}<m+1$.
	%%%%%%%%%%%%
	\item
	\emph{$\M$ is the rotation group $SO(m)$ with the bi-invariant metric determined by $g(E_1,E_2)$ $:=-\frac{1}{2}\text{tr}(E_1E_2)$ for skew-symmetric $E_1, E_2$, where $m>2$}. Assume that, for $S\in\M$, $\phi(S)=-c\,\text{tr}(S_0S)$ with $S_0\in SO(m)$ and a constant $c>0$. Then, the corresponding $\mu_\phi$ is a von Mises-Fisher distribution on $SO(m)$. It can be shown that Hess$^\phi\geqslant-c\,g$. 
	
Recall that the Killing form of $\M$ is $B(E_1,E_2)=(m-2)\text{tr}(E_1E_2)$ and the Ricci curvature Ric$(E_1,E_2)=-\frac{1}{4}B(E_1,E_2)=\frac{m-2}{2}g(E_1,E_2)$. Thus, in this case,
\[\text{Ric}+\text{Hess}^\phi\geqslant\left\{\frac{m-2}{2}-c\right\}g,\]
and so, for the von Mises-Fisher distribution on $SO(m)$, condition \eqref{eqn2} holds with $\kappa>\max\{c-(m-2)/2,0\}$ and assumption (A1) holds if $c<(m-2)/2$.
\end{enumerate}
\end{example}

\section{The distance between coupled diffusions}
\label{distance}
Our approach to define the Stein equation on $\M$ and analyse properties of its solution rests on the construction of a pair of diffusions $(X_t,Y_t)$, and handling of the distance process $\rho(X_t,Y_t)$ between the pair. In particular, we prove exponential contraction of $\rho(X_t,Y_t)$ towards the initial points, and thus extend the approach used by \cite{MG} on $\mathbb R^m$ to the manifold setting. In contrast to the Euclidean setting, since the distance function $(x,y) \mapsto \rho(x,y)$ is not in $\mathcal C^2(\M\times\M)$ if the cut locus of a point in $\M$ is not empty, analysis of the distance process $\rho(X_t,Y_t)$ requires additional care.  

\subsection{When no conjugate points are present in cut loci}
\label{noconjugate}
We first consider the relatively simple situation where there is no conjugate point in the cut locus of any given point in $\M$. In this setting, by modifying the arguments in \cite{WSK1} and \cite{WSK2}, we are able to establish exponential pathwise contraction of distance between the diffusions, aided by a key result given in Lemma 3 in Appendix A of Supplementary Material, which expresses the distance function in terms of finitely many smooth functions in neighbourhoods of cut point, despite it not belonging to $\mathcal C^2(\M\times\M)$. 

Note first that, in terms of a Brownian motion $B_t$ on $\mathbb R^m$ starting from the origin, the It\^o differential equation \eqref{eqn1} with initial condition $X_0=x_0$ is equivalent to   
\begin{eqnarray}
	\begin{array}{rcl}
		\dd_s\!X_t&=&\Xi_t\,\dd_s\!B_t-\frac{1}{2}\nabla\phi(X_t)\,\d t, \qquad X_0=x_0;\\ 
		\dd_s\!\Xi_t&=&H_{\Xi}\dd_s\!X_t,\qquad \Xi(X_0)=\xi_0,
	\end{array}
	\label{eqn1a}
\end{eqnarray}
where $\dd_s$ denotes the Stratonovich differential, $H$ the horizontal lift from $T\M$ to the tangent bundle of the orthonormal frame bundle $\mathcal O(\M)$, where $\xi_0$ sits above $x_0$. For an introduction to horizontal lifts and orthonormal frame bundles, see for example, \cite{KN}. 

\begin{theorem}
	Assume that $\M$ has the property that there is no conjugate point to any given point in $\M$, and that the Bakry-Emery curvature criterion (A1) holds for a constant $\kappa>0$. Then, for any $x_0,y_0\in\M$, there is a pair of coupled diffusions $(X_t,Y_t)$ starting from $(x_0,y_0)$ such that both $X_t$ and $Y_t$ satisfy \eqref{eqn1} and, for any $\ell\geqslant1$,
	\begin{eqnarray}
		\rho(X_t,Y_t)^\ell\leqslant\rho(x_0,y_0)^\ell e^{-\ell\kappa t},\qquad t\geqslant0.
		\label{eqn7a}
	\end{eqnarray}
	\label{thm1a}
\end{theorem}

\begin{proof}
	Consider the map
	\[\hbox{Exp}:\,\,T\M\rightarrow\M\times\M;\quad(x,v)\mapsto(x,\exp_x(v)).\]
	For any $(x,v)\in T\M$, this map provides an intervening geodesic $s\mapsto\exp_x(sv)$, $0\leqslant s\leqslant1$, connecting $x$ and $\exp_x(v)$. The length of this geodesic is at least the distance between $x$ and $\exp_x(v)$. If the interior of this geodesic does not intersect the cut locus of $x$, then it is also a minimal geodesic between its two end points. Denote by $\tilde\Pi_{(x,v)}$ the parallel transport along this \textit{intervening} geodesic from $x$ to $\exp_x(v)$ where, for our purpose, $\tilde\Pi_{(x,v)}$ is taken to be the identity map on $T_x(\M)$ if $x=\exp_x(v)$ even though this may imply a discontinuity.
	
	For any given $(x_0,y_0)\in\M\times\M$, we take $v_0\in T_{x_0}(\M)$ such that 
	\begin{eqnarray}
		\exp_{x_0}(v_0)=y_0\qquad\hbox{ and }\qquad|v_0|=\rho(x_0,y_0). 
		\label{eqn2a}
	\end{eqnarray}
	Under the given assumptions, $y_0$ is not conjugate to $x_0$. Then, if $y_0$ is a cut point of $x_0$, a consequence of the proof of Lemma 3 in Appendix A of Supplementary Material is that there is a neighbourhood $\mathcal N$ of $(x_0,y_0)$ such that $\text{Exp}^{-1}(\mathcal N)$ is a disjoint union of a finite number of open sets on $\TM$ and, restricted to each such set, Exp is a diffeomorphism from that set onto $\mathcal N$. If $y_0$ is not a cut point of $x_0$, then $v_0$ is uniquely determined by $v_0=\exp^{-1}_{x_0}(y_0)$ and a similar result holds with just one component in $\text{Exp}^{-1}(\mathcal N)$. Hence, in particular, $\TM$ is locally a covering space of $\M\times\M$. Within such a neighbourhood $\mathcal N$ of a given $(x_0,y_0)$, we can determine a continuous process $(X_t,V_t)\in\TM$ starting from $(x_0,v_0)$ associated with \eqref{eqn1}, by solving the following coupled diffusions $X_t$ and $Y_t=\exp_{X_t}(V_t)$:
	
	\begin{eqnarray}
		\begin{array}{crl}
			\dd_s\!X_t&=&\Xi_t\,\dd_s\!B_t-\frac{1}{2}\nabla\phi(X_t)\,\d t;\qquad X_0=x_0;\vspace{0.1cm}\\ 
			\dd_s\!Y_t&=&\Upsilon_t\,\dd_s\!B'_t-\frac{1}{2}\nabla\phi(Y_t)\,\d t,\qquad Y_0=y_0;\vspace{0.1cm}\\
			\dd_s\!\Xi_t&=&H_{\Xi}\dd_s\!X_t,\qquad \Xi(X_0)=\xi_0\vspace{0.1cm};\\
			\dd_s\!\Upsilon_t&=&H_{\Upsilon}\dd_s\!Y_t,\qquad\Upsilon(Y_0)=\eta_0\vspace{0.1cm};\\
			\d B'_t&=&(\Upsilon^{-1}_t\tilde\Pi_{X_t,V_t}\Xi_t)\,\d B_t,
		\end{array}
		\label{eqn8}
	\end{eqnarray}
	where, similarly to $\Xi$ and $\xi_0$ for $X$, $\Upsilon$ and $\eta_0$ are respectively a lift of $Y$ to the orthonormal frame bundle $\mathcal O(\M)$ and $\eta_0$ sits above $y_0$. Since $B_t'$ is also a Brownian motion on $\mathbb R^m$, both $X_t$ and $Y_t$ are diffusions satisfying \eqref{eqn1} before they leave $\mathcal N$. 
	
	When $(X_t,Y_t)$ hits the boundary of $\mathcal N$, we can find a neighbourhood $\mathcal N'$ of $(X_t,Y_t)$ satisfying the above properties of $\mathcal N$. Then, allowing $V_t$ to move discontinuously without altering $(X_t,Y_t)$ such that, after the jump, it satisfies \eqref{eqn2a}, we can continue to run $(X_t,Y_t)$ within $\mathcal N'$ so defined. Note that, if $X_{t_0}=Y_{t_0}$ for some $t_0\geqslant0$, then $X_t=Y_t$ for $t\geqslant t_0$.
	
	For $(X_t,Y_t)$ constructed as above, denote by $\tilde\rho(X_t,Y_t)$ the length of the intervening geodesic $\exp_{X_t}(sV_t)$ between $X_t$ and $Y_t=\exp_{X_t}(V_t)$; and write $\gamma_t$ for the unit speed intervening geodesic from $X_t$ to $Y_t$, that is, $\gamma_t(s)=\exp_{X_t}(sV_t/|V_t|)$. Note that $\tilde\rho(X_t,Y_t)$ depends implicitly on the choice of $v_0$, which is not unique when $y_0$ is a cut point of $x_0$. On the other hand, for any given $v_0$ which satisfies \eqref{eqn2a}, $\tilde\rho$ is a smooth function of $(x,y)$ within the neighbourhood $\mathcal N$ chosen as above. However, the change of neighbourhood from $\mathcal N$ to $\mathcal N'$ usually results in a discontinuity for the process $\tilde\rho(X_t,Y_t)$. Nevertheless, $\rho(X_t,Y_t)$ is always continuous and
	\[\rho(X_t,Y_t)\leqslant\tilde\rho(X_t,Y_t),\qquad t\geqslant0,\]
	where the latter becomes an equality immediately after the jump. Hence, to find an upper bound for $\rho(X_t,Y_t)$, it is sufficient to find an upper bound for $\tilde\rho(X_t,Y_t)$. 
	
	To bound $\tilde\rho(X_t,Y_t)$ we may assume, without loss of generality, that $(X_t,Y_t)$ lies in $\mathcal N$ for all $t\geqslant 0$. Write $u_0,u_1,\cdots,u_{m-1}$ for an orthonormal base in $\mathbb R^m$ such that $\Xi_tu_0=\dot\gamma_t(0)$, and, for $i=0,1,\cdots,m-1$, let $v_i=(\Upsilon^{-1}_t\tilde\Pi_{(X_t,V_t)}\Xi_t)u_i$. Then, the It\^o formula for $\tilde\rho(X_t,Y_t)$ is given by
	\begin{eqnarray}
		\begin{array}{rcl}
			\d\tilde\rho(X_t,Y_t)&=&\left(\Xi_tu_0\right)\tilde\rho(X_t,Y_t)\,\d\,\langle u_0,B_t\rangle+\left(\Upsilon_tv_0\right)\tilde\rho(X_t,Y_t)\,\d\,\langle v_0,B'_t\rangle\vspace{0.1cm}\\ 
			&&+\,\,\dfrac{1}{2}\displaystyle\sum_{i=0}^{m-1}\left(\Xi_tu_i+\Upsilon_tv_i\right)^2\tilde\rho(X_t,Y_t)\,\d t\\
			&&+\,\,\dfrac{1}{2}\left\{\langle\nabla\phi(X_t),\dot\gamma_t(0)\rangle-\langle\nabla\phi(Y_t),\dot\gamma_t(\tilde\rho(X_t,Y_t))\rangle\right\}\d t.
		\end{array}
		\label{eqn3a}%{eqn*}
	\end{eqnarray}
	Since $\langle u_0,B_t\rangle=\langle v_0,B'_t\rangle$, since $\left(\Xi_tu_0\right)\tilde\rho(X_t,Y_t)=-\left(\Upsilon_tv_0\right)\tilde\rho(X_t,Y_t)$ and since
	\[\left(\Xi_tu_0\right)^2\tilde\rho(X_t,Y_t)=\left(\Upsilon_tv_0\right)^2\tilde\rho(X_t,Y_t)=\left(\Xi_tu_0\right)\left(\Upsilon_tv_0\right)\tilde\rho(X_t,Y_t)=0,\]   
	\eqref{eqn3a} simplifies to
	\begin{eqnarray}
		\begin{array}{rcl}
			2\d\tilde\rho(X_t,Y_t)&=&\displaystyle\sum_{i=1}^{m-1}\left(\Xi_tu_i+\Upsilon_tv_i\right)^2\tilde\rho(X_t,Y_t)\,\d t \vspace{0.2cm} \\
			&&+\,\,\left\{\langle\nabla\phi(X_t),\dot\gamma_t(0)\rangle-\langle\nabla\phi(Y_t),\dot\gamma_t(\tilde\rho(X_t,Y_t))\rangle\right\}\d t.
		\end{array}
		\label{eqn3b}%{eqn**}
	\end{eqnarray}
	Denote by $J^i_t$ the Jacobi vector field along $\gamma_t$ with $J^i_t(0)=\Xi_tu_i$ and $J^i_t(1)=\Upsilon_tv_i$. Then, since $\tilde\rho$ is smooth under the assumption that $(X_t,Y_t)$ lies in a given neighbourhood of $(x_0,y_0)$, using the second-variation formula (see \cite{CE}), a modification of the argument by \cite{WSK1} shows that the right hand side of \eqref{eqn3b} is given by
	\begin{eqnarray}
		\begin{array}{rcl}
			&&\displaystyle\int_0^{\tilde\rho(X_t,Y_t)}\displaystyle\sum_{i=1}^{m-1}\!\left\{|D_{\dot\gamma_t(s)}(J^i_t(s))|^2\!\!-\!\langle R(J^i_t(s),\dot\gamma_t(s))\dot\gamma_t(s),J^i_t(s)\rangle\right\}\d s\d t \vspace{0.2cm}\\
			&&\quad+\,\,\left\{\langle\nabla\phi(X_t),\dot\gamma_t(0)\rangle-\langle\nabla\phi(Y_t),\dot\gamma_t(\tilde\rho(X_t,Y_t))\rangle\,\d t\right\},
		\end{array}
		\label{eqn3}
	\end{eqnarray}
	where the integral is along $\gamma_t$ and $R$ denotes the curvature tensor of $\M$. 
	
	To analyse the first term of \eqref{eqn3}, we use a modified form of the argument in \cite{CE}, the proof of Lemma 1.21. It shows that, for each $i=1,\ldots,m-1$,
	\begin{eqnarray*}
		&&\int_0^{\tilde\rho(X_t,Y_t)}\left\{|D_{\dot\gamma_t(s)}(J^i_t(s))|^2-\langle R(J^i_t(s),\dot\gamma_t(s))\dot\gamma_t(s),J^i_t(s)\rangle\right\}\d s\\
		&\leqslant&\int_0^{\tilde\rho(X_t,Y_t)}\left\{|D_{\dot\gamma_t(s)}(V^i_t(s))|^2-\langle R(V^i_t(s),\dot\gamma_t(s))\dot\gamma_t(s),V^i_t(s)\rangle\right\}\d s,
	\end{eqnarray*}
	where $V^i_t(s):=(\tilde\Pi_{(X_t,sV_t/|V_t|)}\Xi_t)u_i$. Now, since $V^i_t$ is parallel along $\gamma_t$, it follows that $D_{\dot\gamma_t(s)}(V^i_t(s))=0$. As a consequence, since $\{\dot\gamma_t(s),V^1_t(s),\ldots,V^{m-1}_t(s)\}$ forms an orthonormal base of $T_{\gamma_t(s)}(\M)$ and $\langle R(\dot\gamma_t(s),\dot\gamma_t(s))\dot\gamma_t(s),\dot\gamma_t(s)\rangle\equiv0$, we have
	\begin{eqnarray}
		\begin{array}{rcl}
			&&\displaystyle\int_0^{\tilde\rho(X_t,Y_t)}\displaystyle\sum_{i=1}^{m-1}\left\{|D_{\dot\gamma_t(s)}(J^i_t(s))|^2-\langle R(J^i_t(s),\dot\gamma_t(s))\dot\gamma_t(s),J^i_t(s)\rangle\right\}\d s\\
			&\leqslant&-\displaystyle\int_0^{\tilde\rho(X_t,Y_t)}\displaystyle\sum_{i=1}^{m-1}\langle R(V^i_t(s),\dot\gamma_t(s))\dot\gamma_t(s),V^i_t(s)\rangle\,\d s\\
			&=&-\displaystyle\int_0^{\tilde\rho(X_t,Y_t)}\hbox{Ric}(\gamma_t(s))(\dot\gamma_t(s),\dot\gamma_t(s))\,\d s.
		\end{array}
		\label{eqn4}
	\end{eqnarray}

	For the remaining two terms of \eqref{eqn3}, we note that
	\begin{eqnarray*}
		\frac{\d}{\d s}\langle\nabla\phi(\gamma_t(s)),\dot\gamma_t(s)\rangle
		&=&\langle D_{\dot\gamma_t(s)}(\nabla\phi(\gamma_t(s))),\dot\gamma_t(s)\rangle+\langle\nabla\phi(\gamma_t(s)),D_{\dot\gamma_t(s)}\dot\gamma_t(s)\rangle\\
		&=&\langle D_{\dot\gamma_t(s)}(\nabla\phi(\gamma_t(s))),\dot\gamma_t(s)\rangle\\
		&=&\hbox{Hess}^\phi(\dot\gamma_t(s),\dot\gamma_t(s)),
	\end{eqnarray*}
	as $\gamma_t$ is a geodesic. From this, we deduce that
	\begin{eqnarray}
		\begin{array}{rcl}
			\langle\nabla\phi(Y_t),\dot\gamma_t(\tilde\rho(X_t,Y_t))\rangle-\langle\nabla\phi(X_t),\dot\gamma_t(0)\rangle \vspace{0.2cm}
			&=&\displaystyle\int_0^{\tilde\rho(X_t,Y_t)}\hbox{Hess}^\phi(\dot\gamma_t(s),\dot\gamma_t(s))\,\d s.
		\end{array}
		\label{eqn5}
	\end{eqnarray}

	Thus, under the Bakry-Emery curvature criterion (A1) condition, \eqref{eqn3}, \eqref{eqn4} and \eqref{eqn5} together give that
	\begin{eqnarray}
		\begin{array}{rcl}
			2\d\tilde\rho(X_t,Y_t)\vspace{0.2cm}
			&\leqslant&-\displaystyle\int_0^{\tilde\rho(X_t,Y_t)}\!\!\left\{\hbox{Ric}(\dot\gamma_t(s),\dot\gamma_t(s))\!+\!\hbox{Hess}^\phi(\dot\gamma_t(s),\dot\gamma_t(s))\right\}\d s\,\d t\vspace{0.2cm}\\
			&\leqslant&-\,\,2\kappa\tilde\rho(X_t,V_t)\,\d t.
		\end{array}
		\label{eqn6}
	\end{eqnarray}
	Now, for any $\ell\geqslant1$, it follows from \eqref{eqn6} that 
	\[\d\left(\tilde\rho(X_t,Y_t)^\ell\right)\leqslant-\ell\,\kappa\,\tilde\rho(X_t,Y_t)^\ell\d t,\]
	so that
	\begin{eqnarray*}
		e^{\ell\kappa t}\tilde\rho(X_t,Y_t)^\ell&=&\tilde\rho(X_0,Y_0)^\ell+\int_0^t e^{\ell\kappa s}\left\{\ell\kappa\,\tilde\rho(X_s,Y_s)^\ell\d s+\d\left(\tilde\rho(X_s,Y_s)^\ell\right)\right\}\\
		&\leqslant&\tilde\rho(X_0,Y_0)^\ell.
	\end{eqnarray*}  
	Finally, by recalling that $\tilde\rho(X_0,Y_0)=\rho(X_0,Y_0)$, we have
	\[\rho(X_t,Y_t)^\ell\leqslant\tilde\rho(X_t,Y_t)^\ell\leqslant\rho(X_0,Y_0)^\ell e^{-\ell\kappa t}\]
	as required.
\end{proof}

\subsection{When conjugate points are present in cut loci}
\label{conjugate}
When conjugate points are present in cut loci in $\M$, the construction of a pair of diffusions in the proof of Theorem \ref{thm1a} fails at such points. More precisely, if $y_0$ is a (first) conjugate point of $x_0$ along the geodesic $\exp_{x_0}(sv)$, which also lies in the cut locus of $x_0$, then $D\exp_{x_0}(v)$ is singular. This means that it would be impossible to find a neighbourhood $\mathcal N$ of $(x_0,y_0)$ that has the properties described above following \eqref{eqn2a}. In particular, it would be impossible to find a subset of $\TM$, as specified there, such that Exp is a diffeomorphism from that subset onto $\mathcal N$. It is evident from the proof of Theorem \ref{thm1a} that the existence of such a diffeomorphism offers a way to couple $(X_{x,t},Y_{y,t})$ at, and beyond, cut points.

Nevertheless, we now show that it is still possible to construct a pair of diffusions on $\M$ with properties that ($\mathrm{i}$) they both satisfy \eqref{eqn1} and ($\mathrm{ii}$) the \emph{expected} distance between them contracts at least exponentially. This relies on a generalisation of the technique used in Theorem 5 of \cite{WSK2} to deal with the presence of conjugate points. In the non-conjugate part of the cut locus of $\M$ analysis proceeds as with Theorem \ref{thm1a}. To warn us of when the diffusions get close to the first conjugate locus, we use the operator $L_\phi$, and monitor the value of its action on the distance function $\rho$; this value decays towards $-\infty$ when the points approach the first conjugate locus.  Effectively, we determine a neighbourhood $N_{2\delta}\subset \M \times \M$ of the first conjugate locus in $\M\times\M$ for a constant $\delta$ that depends on $\kappa$ and the injectivity radius of $\M$. Once the coupled diffusions enter $\bar N_{2\delta}$, the closure of $N_{2\delta}$, we decouple them, run independent diffusions until they hit $\M\setminus N_\delta$, where $N_\delta\supset N_{2\delta}$, and then return to coupling again. 

We first need two preliminary results before stating and proving the main result in this section. Observe that the set
\begin{eqnarray*}
	\tilde{\mathcal E}:=\{(x,v)\in\TM\,\mid&\,\hbox{the geodesic }\exp_x(sv),\,0\leqslant s\leqslant1,\\
	&\,\hbox{contains no conjugate point of }x\}
\end{eqnarray*}
is an open set in $\TM$. The map Exp$:(x,v)\rightarrow(x,\exp_x(v))$ maps $\tilde{\mathcal E}$ surjectively to its image
\begin{eqnarray}
	\begin{array}{rl}
		\mathcal E:=\{(x,y)\in\M\times\M\,\mid&\hbox{there is a geodesic from $x$ to $y$}\\
		&\hbox{containing no conjugate point}\}.
	\end{array}
	\label{eqn2b}
\end{eqnarray}
Then, the construction \eqref{eqn8} of $(X_t,Y_t)$ can be applied to the case when the starting point $(x_0,y_0)$ is in $\mathcal E$ and it remains valid until the first exit of $(X_t,V_t)$ from $\tilde{\mathcal E}$. We now modify the construction by \cite{WSK2}: combine the coupled diffusions $(X_t,Y_t)$ defined by \eqref{eqn8}, while the corresponding $(X_t,V_t)$ is not too close to the boundary of $\tilde{\mathcal E}$, with $X_t$, $Y_t$ evolving independently.

For this, we first need a result on the distance function of two independent diffusions on $\M$ specified by \eqref{eqn1}. Lemma 3 in Appendix A of the Supplementary Material ensures the following property of $\rho(x,y)$ on neighbourhoods of the cut locus
\[\mathcal C:=\{(x,y)\in\M\times\M\mid y\hbox{ lies in the cut locus of }x\}\]
of $\M\times\M$: there is a set $\mathcal C_0\subset\mathcal C$ such that 
\begin{enumerate}[label=(\roman*), leftmargin=*]
	%\item[($i$)] 
	\item $\mathcal C_0$ contains the (first)-conjugate part of $\mathcal C$;%\supset\mathcal C\cap\mathcal E^c$;
	%\item[($ii$)] 
	\item for any $(x,y)\in\mathcal C\setminus\mathcal C_0$, there is a neighbourhood $\mathcal N$ of $(x,y)$ in $\M\times\M$ and two smooth functions $\varrho_1$ and $\varrho_2$ on $\mathcal N$ such that 
	\[\rho(x',y')=\min\{\varrho_1(x',y'),\varrho_2(x',y')\},\qquad\forall(x',y')\in\mathcal N.\] 
\end{enumerate}
Since the (first)-conjugate part of $\mathcal C$ has co-dimension 2 in $\M\times\M$ (see \cite{BL1}), the result of that Lemma also implies that $\mathcal C_0$ can be chosen to have co-dimension 2. Also, similarly to the argument at the beginning of the proof of Theorem 1,  $\mathcal N$ in $(ii)$ above can be chosen such that $\text{Exp}^{-1}(\mathcal N)$ is a disjoint union of two open sets $\mathcal V_1$, $\mathcal V_2$ in $\TM$ and, restricted to each $\mathcal V_i$, Exp is a diffeomorphism from that set to $\mathcal N$. Then, the smooth function $\varrho_i(x',y')$ constructed in the proof of Lemma  3 in Appendix A of the Supplementary Material is in fact the length of the geodesic from $x'$ and $y'$, the initial tangent vector $v_i$ to which lies in $\mathcal V_i$. That is, using our notation for the length of intervening geodesics, we have $\varrho_i(x',y')=\tilde\rho(x',\exp_{x'}(v_i))$. This leads to the following generalisation of Theorem 5 of \cite{WSK2} and of Theorem 3 of \cite{BL1}. The proof of this generalisation is a slight modification of the proof for Theorem 3 of \cite{BL1} (see also \cite{LB} for more detailed derivations), and we hence omit it here.

\begin{lemma}
	Suppose that $X_t$ and $Y_t$ are independent diffusions on $\M$, both satisfying \eqref{eqn1}. Then, the distance $\rho(X_t,Y_t)$ is a semimartingale and, before the first time that $X_t=Y_t$,
	\[\d\rho(X_t,Y_t)=\sqrt2\d B_t+\frac{1}{2}\left\{\mathcal L_{\phi,1}\rho(X_t,Y_t)+\mathcal L_{\phi,2}\rho(X_t,Y_t)\right\}\d t-\d L_t,\]
	where $B_t$ is a Brownian motion on $\mathbb R$; $L$ is a non-decreasing process that is locally constant outside $\mathcal C$; and, for fixed $x_0$ and $x\not=x_0$,
	\[\mathcal L_{\phi,1}\rho(x,x_0)\!:=\!\!\left\{
	\begin{array}{ll}
		0&\hbox{if }(x,x_0)\in\mathcal C_0;\\
		\!\!\!\dfrac{1}{2}\{L_\phi\tilde\rho(\exp_{x_0}\!(v_1),x_0)\!+\!L_\phi\tilde\rho(\exp_{x_0}\!(v_2),x_0)\}&\hbox{if  }(x,x_0)\in\mathcal C\!\setminus\!\mathcal C_0;\\
		L_\phi\rho(x,x_0)&\hbox{otherwise},
	\end{array}
	\right.\]
	and $\mathcal L_{\phi,2}\rho$ is similarly defined with respect to the second argument of $\rho$, and where the operator $L_\phi$ is defined by $L_\phi=\frac{1}{2}\{\Delta-\langle\nabla\phi,\,\nabla\rangle\}$.
	\label{lem5}
\end{lemma}

To detect that the coupled $(X_t,Y_t)$, constructed by \eqref{eqn8}, is close to the boundary of $\mathcal E$ and to control the independent diffusions $X_t$ and $Y_t$, we need the following generalisation of a geometric description (see \cite{WSK2}), wherein we replace the Laplacian operator considered there with $\mathcal L_{\phi}$, and replace the lower bound constant $c$ determining the set $\mathcal O_c$ (which was denoted by $U_c$ by \cite{WSK2}) by $c\rho(x,y)$. Since $\phi$ is in $\mathcal C^2(\M)$, the proof for our result is analogous to that for the lemma in \cite{WSK2}, and we omit it here.

\begin{lemma}
	For any $c>0$,
	\[\mathcal O_c\subset\bar{\mathcal O}_c\subset\tilde{\mathcal E},\]
	where 
	\[\mathcal O_c:=\{(x,v)\in\tilde{\mathcal E}\mid\mathcal L_{\phi,1}\tilde\rho(x,\exp_x(v))+\mathcal L_{\phi,2}\tilde\rho(x,\exp_x(v))>-2c\rho(x,\exp_x(v))\}\]
	and, as before, $\tilde\rho(x,exp_x(v))$ denotes the length of the intervening geodesic $\gamma(t)=\exp_x(tv)$, $0\leqslant t\leqslant 1$. 
	\label{lem6}
\end{lemma}

We are now ready to prove the following result for Riemannian manifolds $\M$ with non-empty conjugate locus (e.g., spheres), which is weaker than Theorem \ref{thm1a} in that the exponential contraction between the diffusions towards their initial points is in expection and not pathwise. 

\begin{theorem}
	Assume that the  Bakry-Emery curvature criterion (A1) holds for a constant $\kappa>0$. Then, for any $\ell\geqslant1$ and for any $x_0,y_0\in\M$, there is a pair of diffusions $(X_t,Y_t)$ starting from $(x_0,y_0)$ such that both $X_t$ and $Y_t$ satisfy \eqref{eqn1} and
	\begin{eqnarray}
		\E{\rho(X_t,Y_t)^\ell}\leqslant\rho(x_0,y_0)^\ell e^{-\ell\kappa t},\qquad t\geqslant0.
		\label{eqn7}
	\end{eqnarray}
	\label{thm4}
\end{theorem}

Note that, unlike the result of Theorem \ref{thm1a}, the $(X_t,Y_t)$ constructed here will depend on $\ell$. 
\begin{proof}
	Let $\kappa>0$ be the constant in Bakry-Emery curvature criterion (A1)  For given $\ell\in[1,n]$, fix $\delta_n>0$ sufficiently large such that
	\begin{enumerate} [label=(\roman*), leftmargin=*]
		\item $\delta_n>\kappa+4(n-1)/r_0^2$, where $r_0>0$ is the minimum of the injectivity radius and a fixed positive constant $r_0'$ say;
		\item $O_{\delta_n}\supset\{(x,y)\in\M\times\M\mid\rho(x,y)<r_0/2\}$, where $O_\delta=\text{Exp}\left(\mathcal O_\delta\right)$ and where $\mathcal O_c$ is the subset of $\TM$ as defined in Lemma \ref{lem6} above.
	\end{enumerate}

	We now construct diffusions $X_t$ and $Y_t$, both satisfying \eqref{eqn1}, as follows. For given $(x_0,y_0)\in\M\times\M$, if there is a minimal geodesic between them which contains no conjugate point, we construct diffusions $X_t$ and $Y_t$ by solving \eqref{eqn8} beginning at $(x_0,y_0)$. By allowing the corresponding $(X_t,V_t)$ to jump if necessary, as commented following the construction \eqref{eqn8}, we continue such a construction for $(X_t,Y_t)$ until the first time that $(X_t,V_t)$ leaves $\mathcal O_{2\delta_n}$. Suppose that $(X_t,V_t)$ leaves $\mathcal O_{2\delta_n}$ at time $\tau$. We then consider all minimal geodesics between $X_\tau$ and $Y_\tau$ containing no conjugate point and, if possible, choose one for which the corresponding $(X_\tau,V_\tau)$ lies in $\bar{\mathcal O}_{\delta_n}$. We then repeat the construction as before with the chosen new starting point. This iterated construction continues until the choice of such $(X_\tau,V_\tau)$ in $\bar{\mathcal O}_{\delta_n}$ is no longer possible.
	
	If it is not possible initially to choose a minimal geodesic containing no conjugate point, or if at some stage a choice of the above $(X_\tau,V_\tau)$ in $\bar{\mathcal O}_{\delta_n}$ is impossible, then we continue the construction of $X_t$ and $Y_t$ by evolving them independently until $(X_t,Y_t)$ hits $\bar O_{\delta_n}$. 
	
	To show that the required result holds for $(X_t,Y_t)$ constructed in such a way, it is sufficient by Theorem 1 to restrict to the case when $X_t$ and $Y_t$ evolve independently. Then, $(X_t,Y_t)$ is not in $\bar O_{\delta_n}$. Recalling that a co-dimension 2 set in $\M\times\M$ is a polar set of a non-degenerate diffusion on $\M\times\M$ it follows from Lemmas \ref{lem5} and \ref{lem6} and from the choice of $\delta_n$ that 
	\begin{eqnarray*}
		&&\d\left(e^{\ell\kappa t}\rho(X_t,Y_t)^\ell\right)\\
		&\leqslant&\d M_t+\ell\kappa e^{\ell\kappa t}\rho(X_t,Y_t)^\ell\d t\\
		&&+\,\,\frac{1}{2}\ell e^{\ell\kappa t}\rho(X_t,Y_t)^{\ell-1}\left\{\mathcal L_{\phi,1}\rho(X_t,Y_t)+\mathcal L_{\phi,2}\rho(X_t,Y_t)\right\}\d t\\
		&&+\,\,\ell(\ell-1)\,e^{\ell\kappa t}\rho(X_t,Y_t)^{\ell-2}\d t\\
		&\leqslant&\d M_t+\ell e^{\ell\kappa t}\rho(X_t,Y_t)^\ell\left\{\kappa-\delta_n+(\ell-1)\rho(X_t,Y_t)^{-2}\right\}\d t\\
		&\leqslant&\d M_t+\ell e^{\ell\kappa t}\rho(X_t,Y_t)^\ell\left\{\kappa-\delta_n+4(n-1)/r_0^2\right\}\d t\\
		&\leqslant&\d M_t,
	\end{eqnarray*}
	where $M_t$ is a martingale. Hence, we have $\E{\rho(X_t,Y_t)^\ell}\leqslant\rho(x_0,y_0)^\ell e^{-\ell\kappa t}$ as required.
\end{proof}

\begin{remark}
In the literature, there are several ways to construct couplings for proving the existence of contractivity. For example, in the curvature setting, the framework of weighted Riemannian manifolds is now part of a broader one for CD-spaces (see e.g., \cite{KTS, KTS2}). In this context, the existence of contractive couplings was treated by \cite{KK, VRTS}. In particular, the Kuwada duality theorem (see \cite{KK}, Theorem 2.2), in conjunction with the implication of contractivity of the heat flow under Curvature-Dimension condition, implies the existence of a contractive coupling such as in the proof of Corollary 1 in \cite{VRTS}. The coupling we construct here, in addition to proving the required contractivity, will also be employed in the Supplementary Material to study certain stochastic vector fields along the paths $X_{x,t}$ and $Y_{y,t}$, which play important roles in obtaining the Stein factors.
\end{remark}

\section{Solution to the Stein equation and Stein factors}
\label{Stein_equation}
\label{equation}
We are now ready to turn our attention to the \emph{Stein equation}
\begin{equation}
	\label{stein_eq}
h(x)-\E{h(X)}=L_\phi f_h(x),
\end{equation}
where $h$ belongs to a suitable class of real-valued test functions on $\M$.  Using the distance process $\rho(X_{x,t},Y_{y,t})$ for a pair of diffusions $(X_{x,t}, Y_{y,t})$ constructed above, in this Section we determine the solution $f_h$ to the Stein equation \eqref{stein_eq} and examine its properties.

\subsection{The solution $f_h$}
Let
\begin{equation}
\label{H_0}
\mathcal H_0:=\{h\in\mathcal C_0(\M)\mid h\hbox{ is Lipschitz with }C_0(h) <\infty\}.
\end{equation}
\begin{proposition}
	Let $\M$ be a complete and connected Riemannian manifold. Assume that the Bakry-Emery curvature criterion (A1) holds for a constant $\kappa>0$ and that $X$ is a random variable on $\M$ with distribution $\mu_\phi$ such that $\E{\rho(X,x)}<\infty$ for some $x\in\M$.
	For every $h \in \mathcal H_0$ the function 
	\begin{equation}
		f_h(x):=\int_0^\infty\left\{\E{h(X)}-\E{h(X_{x,t})}\right\}\d t
		\label{eqn10}
	\end{equation}
	is (i) well-defined; (ii) Lipschitz with constant $C_0(f_h)\leqslant C_0(h)/\kappa$.
	\label{prop5}
\end{proposition}

\begin{remark}
If $\M=\mathbb R^m$, $\text{Ric}(u, u')+\hbox{Hess}^\phi(u,u')=\hbox{Hess}^\phi(u,u')$. Thus, Proposition \ref{prop5}(ii) recovers the corresponding result in \cite{MG}, as the constant $2\kappa$ here corresponds to constant $k$ there. Moreover, the result of Proposition \ref{prop5}(ii) is equivalent to that of Proposition 6.1 in \cite{JT}.
 \label{remark2}
\end{remark}

\begin{proof}
	Let $(X_{x,t},Y_{y,t})$ be the pair of diffusions in Theorem \ref{thm4} with $\ell=1$, starting from $(x,y)$. Then, both $X_{x,t}$ and $Y_{y,t}$ satisfy \eqref{eqn1}. Since $\mu_\phi$ is the invariant measure for $Y_t$, using the Lipschitz property of $h$ and Theorem \ref{thm4},
	\begin{eqnarray*}
		&&\left|\int_0^\infty\left\{\E{h(X)}-\E{h(X_{x,t})}\right\}\d t\right|\\
		&=&
		\left|\int_0^\infty\int_{\M}\left\{\E{h(Y_{y,t})}-\E{h(X_{x,t})}\right\}\d\mu_\phi(y)\d t\right|\\
		&\leqslant&C_0(h)\int_0^\infty\int_{\M}\E{\rho(X_{x,t},Y_{y,t})}\,\d\mu_\phi(y)\,\d t\\
		&\leqslant&C_0(h)\E{\rho(X,x)}\int_0^\infty e^{-\kappa t}\d t<\infty.
	\end{eqnarray*}
	This proves that $f_h$ is well-defined. Now, for any $x,y\in\M$,
	\begin{eqnarray*}
		|f_h(y)-f_h(x)|&\leqslant&\int_0^\infty\left|\E{h(Y_{y,t})}-\E{h(X_{x,t})}\right|\d t\\
		&\leqslant&C_0(h)\int_0^\infty\E{\rho(X_{x,t},Y_{y,t})}\d t\\
		&\leqslant&C_0(h)\rho(x,y)\int_0^\infty e^{-\kappa t}\d t=\frac{1}{\kappa}C_0(h)\rho(x,y).
	\end{eqnarray*}
\end{proof}

The next result shows that the function $f_h$ defined by \eqref{eqn10} solves the Stein equation for the probability measure $\mu_\phi$. 

\begin{theorem}
	Assume that $\M$ is a complete and connected Riemannian manifold and that Bakry-Emery curvature criterion (A1) holds for a constant $\kappa>0$. Let $X$ be a random variable on $\M$ with distribution $\mu_\phi$ such that $\E{\rho(X,x)}<\infty$ for some $x\in\M$. For $h\in \mathcal H_0$, the function $f_h$ in \eqref{eqn10} solves the Stein equation \eqref{stein_eq}. 
	\label{thm2}
\end{theorem}

\begin{remark}
\label{remark3}
When $\M=\mathbb R^m$ this result recovers the result by \cite{MG}; in particular, $\E{L_\phi f_h(X)}=0$. On the other hand, the Bakry-Emery curvature criterion (A1) implies certain restrictions on the probability measures to which we can apply Theorem \ref{thm2}. For example, as noted in Example \ref{examples}(i), one cannot apply it to von Mises-Fisher distributions on the circle.
In this case, using direct integration by parts, for probability measures $\mu_\phi$ with $X\sim \mu_\phi$ on $S^1$, the function 
	\[g_h(x)=c(\phi) e^{\phi(x)}\left\{a+\int_{-\pi}^x(h(y)-\E{h(X)})\d\mu_\phi(y)\right\},\] for a constant $a$,
	solves the Stein equation $h(x)-\E{h(X)}=g'_h(x)-\phi'(x)\,g_h(x)$ associated with first-order Stein operator $A_\phi g-\phi g=g'-\phi'\,g$ (see \cite{ALewis}). 
\end{remark}

\begin{proof}
	Let $X_{x,t}$ be a diffusion starting from $x$ and satisfying \eqref{eqn1}. Since the corresponding semigroup $\{P_t\mid t\geqslant 0\}$ is strongly continuous on $\mathcal C_0(\M)$ and $L_\phi$ is the infinitesimal generator of $X_{x,t}$, we have
	\[(P_th)(x)-h(x)=%\int_0^tL_\phi\left(\E{h(X_{x,s})}\right)\d s=
	L_\phi\left(\int_0^t\E{h(X_{x,s})}\d s\right)\] 
	for $h\in\mathcal C_0(\M)$ \citep[Prop. 1.5]{EK}. However, for $\tilde h(x)=h(x)+a$ where $a\in\mathbb R$, $\tilde h(x)-\E{\tilde h(X)}=h(x)-\E{h(X)}$. Then, by taking $a=\E{h(X)}$ and noting $L_\phi(a)=0$, we can also write the above as
	\begin{eqnarray}
		(P_th)(x)-h(x)=-L_\phi\left(\int_0^t\left\{\E{h(X)}-\E{h(X_{x,s})}\right\}\d s\right).
		\label{eqn12}
	\end{eqnarray} 

	Now, take $(X_{x,t},Y_{y,t})$ to be the pair of diffusions, starting from $(x,y)$, as Theorem \ref{thm4} with $\ell=1$. Since $Y_t$ satisfies \eqref{eqn1}, the fact that $\mu_\phi$ is the invariant measure of $Y_t$ gives that 
	\begin{eqnarray*}
		&&\big|\E{h(X)}-(P_th)(x)\big|=\left|\int_{\M}\left\{\E{h(Y_{y,t})}-\E{h(X_{x,t})}\right\}\d\mu_\phi(y)\right|\\
		&\leqslant& C_0(h)\int_{\M}\E{\rho(Y_{y,t},X_{x,t})}\d\mu_\phi(y)\leqslant C_0(h)\E{\rho(X,x)}e^{-\kappa t},
	\end{eqnarray*}
	where the last inequality follows from Theorem \ref{thm4} and where $C_0(h)$ is the Lipschitz constant for $h$. Thus, 
	\[\lim_{t\rightarrow\infty}(P_th)(x)=\E{h(X)}.\]
	On the other hand, the result of Theorem \ref{thm4} implies that we may apply the Dominated Convergence Theorem to obtain that, as $t\rightarrow\infty$, the right hand side of \eqref{eqn12} tends to $-L_\phi f_h(x)$, so that 
	\[h(x)-\E{h(X)}=L_\phi f_h(x)\]
	as required. 
\end{proof}

\subsection{Stein factors}
In the literature, Stein factors refer to bounds on solutions $f_h$ of the Stein equation \eqref{stein_eq}. 
A direct consequence of Proposition \ref{prop5} and Theorem \ref{thm2} is that $f_h$ defined by \eqref{eqn10} is differentiable and $Df_h$ is bounded.

\begin{proposition}
	Under the conditions of Theorem $\ref{thm2}$, $Df_h$ exists and 
	\[\sup_{x\in\M}\|Df_h(x)\|_{op}\leqslant C_0(h)/\kappa\]
	where $f_h$ is defined by \eqref{eqn10}.
	\label{cor1}
\end{proposition}

We will see later in Section \ref{wasserstein_bound} that the bound on $Df_h$ given above suffices to bound the Wasserstein distance between the probability measure $\mu_\phi$ and another $\mu_\psi \propto e^{-\psi}$ . However, for bounding more general integral (semi-)metrics, bounds on first- and second-order derivatives of $f_h$, known as Stein factors, are needed.

Accordingly, denote by $\text{Ric}^\sharp_\phi$ the tensor equivalent to $\text{Ric}+\hbox{Hess}^\phi$ in the sense that, for any $x\in\M$, and for any $u,u'\in T_x(\M)$ 
\begin{eqnarray}
	\langle\text{Ric}^\sharp_\phi(u),\,u'\rangle=\text{Ric}(u,u')+\hbox{Hess}^\phi(u,u').
	\label{eqn12d}
\end{eqnarray}
Recall that (see \cite{BO})
\begin{eqnarray}
	\hbox{Hess}^\phi(u,u')=\langle D_u(\nabla\phi),\,u'\rangle\thickspace,
	\label{eqn12f}
\end{eqnarray}
and that, in terms of a (local) frame field $e_1,\cdots,e_m$,
\[\text{Ric}(u,u')=\sum_{i=1}^m\langle R(u,e_i)e_i,\,u'\rangle,\]
where $R$ denotes the Riemannian curvature tensor. Thus, it is possible to express $\text{Ric}^\sharp_\phi$ explicitly in terms of the frame field as
\begin{eqnarray}
	\text{Ric}^\sharp_\phi(u)=\sum_{i=1}^mR(u,e_i)e_i+D_u(\nabla\phi).
	\label{eqn12e}
\end{eqnarray}
We can define the Lipschitz constant for $\text{Ric}^\sharp_\phi$ in a similar way to the definition of the Lipschitz constant given in \eqref{tensor_norm}.  Let
\begin{equation}
	\label{H_1}
	\mathcal H_1:=\{h\in\mathcal C_0(\M) \cap \mathcal C^1(\M) \mid C_0(h) < \infty, \thickspace C_1(h) < \infty\}.
\end{equation}

\begin{proposition}
	Assume that the conditions of Theorem $\ref{thm2}$ hold. Assume further that ${\rm Ric}^\sharp_\phi$ is Lipschitz with finite Lipschitz constant $L({\rm Ric}^\sharp_\phi)$. For every $h \in \mathcal H_1$ with $f_h$ defined in \eqref{eqn10},  $Df_h$ is Lipschitz with constant 
	\[C_1(f_h)\leqslant C_1(h)\frac{1}{2\kappa}+C_0(h)\frac{L({\rm Ric}^\sharp_\phi)}{2\kappa^2}.\] 
	\label{prop6a}
\end{proposition}

\begin{remark}
\label{remark4}
	 As noted in Remark \ref{remark2}, if $\M=\mathbb R^m$, $\langle\text{Ric}^\sharp_\phi(u),\,u'\rangle=\hbox{Hess}^\phi(u,u')$. Then, since $\hbox{Hess}^\phi=D^2\phi$, $L(\text{Ric}^\sharp_\phi)=C_2(\phi)$. Thus, Proposition \ref{prop6a} recovers the corresponding result in \cite{MG}.
On the other hand, the result of Proposition \ref{prop6a} differs from the corresponding Proposition 6.2 in \cite{JT}: in theirs, the relationship between the constant $c_1$ obtained and those given in the assumptions is not specified; using our notation, the upper bound for $C_1(f_h)$ there would depend only on $C_0(h)$ while ours depends on both $C_0(h)$ and $C_1(h)$.
\end{remark}

\begin{proof}
	The proof uses Lemmas 4 and 5 given in Appendix B of Supplementary Material. For any $x\in\M$ and $v\in T_x(\M)$, consider the vector field $v^x_t$ along the path $X_{x,t}$ which solves the differential equation
	\begin{eqnarray}
		\frac{Dv^x_t}{\d t}=-\frac{1}{2}\text{Ric}^\sharp_\phi(v^x_t)
		\label{eqn12a}
	\end{eqnarray}
	with $v^x_0=v$, where $X_{x,t}$ is the solution to \eqref{eqn1}. It is known that, for any fixed $t>0$  and under the given condition for $h$, $N_s=D\E{h(X_{X_{x,s},t-s})}(v^x_s)$ is a local martingale for $0\leqslant s\leqslant t$ (see \cite{AT}). Since
	\[|N_s|\leqslant\|D\E{h(X_{X_{x,s},t-s})}\|_{op}|v^x_s|\leqslant C_0(h)\,|v^x_s|,\] 
	using Lemma 4 (Appendix B of Supplementary Material) with $q=1$, we see that $\E{|N_s|}<\infty$. Hence, $N_s$ is in fact a martingale on $[0,t]$, and so $\E{N_0}=\E{N_t}$, which in turn gives  
	\[D\E{h(X_{x,t})}(v)=\E{Dh(X_{x,t})(v^x_t)}.\]
	(See also \citet[Theorem 11.2]{JT}, where the $Z$ there corresponds to $-2\nabla\phi$ here.) Thus, from the definition of $f_h$, the Dominated Convergence Theorem and Theorem \ref{thm4}, it follows that, for any $v\in T_x(\M)$,
	\begin{eqnarray}
		Df_h(x)(v)=\int_0^\infty D\E{h(X_{x,t})}(v)\d t=\int_0^\infty\E{Dh(X_{x,t})(v^x_t)}\d t.
		\label{eqn10a}
	\end{eqnarray}

	Now, consider the pair of diffusions $(X_{x,t},Y_{y,t})$, starting from $(x,y)$, in Theorem \ref{thm4} with $\ell=2$. First, by applying the H\"older inequality, Theorem \ref{thm4} and Lemma 4 (Appendix B of Supplementary Material), we have that
	\begin{eqnarray}
		\begin{array}{rcl}
			&&\E{\left|\left(Dh(X_{x,t})-\Pi_{\gamma^{\vphantom{A}}_{X_{x,t},Y_{y,t}}}Dh(Y_{y,t})\right)(v^x_t)\right|}\vspace{0.1cm}\\
			&\leqslant&C_1(h)\E{\rho(X_{x,t},Y_{y,t})\,|v^x_t|}\vspace{0.1cm}
			\leqslant C_1(h)\rho(x,y)\,|v|\,e^{-2\kappa t}.
		\end{array}
		\label{eqn10d}
	\end{eqnarray}
	Moreover, writing $v^y_t$ for the solution of \eqref{eqn12a} with the underlying path $X_{x,t}$ replaced by $Y_{y,t}$ and with the initial condition $v^y_0=\Pi_{\gamma_{x,y}}(v)$, and denoting $\Pi_{\gamma^{\vphantom{A}}_{X_{x,t},Y_{y,t}}}(v^x_t)$ by $\tilde v^x_t$, we also have 
	\begin{eqnarray}
		\E{\left|\left(Dh(Y_{y,t})(\tilde v^x_t-v^y_t)\right)\right|}\leqslant C_0(h)\E{\left|\tilde v^x_t-v^y_t\right|}
		\leqslant C_0(h)\dfrac{L(\text{Ric}^\sharp_\phi)}{2\kappa}\rho(x,y)\,|v|\,e^{-\kappa t},
		\label{eqn10e}
	\end{eqnarray}
	where the second inequality follows from Lemma 5 (Appendix B of Supplementary Material) with $q=1$.
	
	Finally, noting that $\Pi_{\gamma_{x,y}}(Df_h(y))(v)=Df_h(y)(\Pi_{\gamma_{x,y}}(v))$, together with \eqref{eqn10a}, \eqref{eqn10d} and \eqref{eqn10e}, implies that
	\begin{eqnarray*}
	|(Df_h(x)-\Pi_{\gamma_{x,y}}Df_h(y))(v)|\vspace{0.1cm}
		&=&|Df_h(x)(v)-Df_h(y)(\Pi_{\gamma_{x,y}}(v))|\\
		&\leqslant&\int_0^\infty\left|\E{Dh(X_{x,t})(v^x_t)-Dh(Y_{y,t})(v^y_t)}\right|\d t\\
		&\leqslant&
		\int_0^\infty\E{\left|\left(Dh(X_{x,t})-
			\Pi_{\gamma^{\vphantom{A}}_{X_{x,t},Y_{y,t}}}Dh(Y_{y,t})\right)(v^x_t)\right|}\d t\\
		&&
		+\int_0^\infty\E{\left|Dh(Y_{y,t})(\tilde v^x_t-v^y_t)\right|}\d t\\
		&\leqslant&\left\{C_1(h)\frac{1}{2\kappa}+C_0(h)\frac{L(\text{Ric}^\sharp_\phi)}{2\kappa^2}\right\}\rho(x,y)|v|,
	\end{eqnarray*}
	i.e. $Df_h$ is Lipschitz with the required constant.
\end{proof}

The argument in Remark \ref{remark4} regarding the case when $\M=\mathbb R^m$ can be extended to the case when $\M$ has constant Ricci curvature, which implies that the bounds in \cite{MG} continue to hold for such $\M$. This gives the following corollary.
 
\begin{corollary}
	Assume that the conditions of Theorem $\ref{thm2}$ hold. Assume further that $\M$ is {\rm Ric} flat and $\phi$ has finite Lipschitz constant $C_2(\phi)$. Then, for every $h \in \mathcal H_1$ and $f_h$ as defined in \eqref{eqn10},  $Df_h$ is Lipschitz with constant 
	\[C_1(f_h)\leqslant C_1(h)\frac{1}{2\kappa}+C_0(h)\frac{C_2(\phi)}{2\kappa^2}.\] 
	\label{cor3}
\end{corollary}

The curvature of the manifold plays a more explicit role in the Lipschitz constant for $D^2f_h$. To see this, define the tensor $\d^\star R$ by
\[\d^\star R(u,v)=-\hbox{tr}D_.R(\cdot, u)v.\]
Then $\d^\star R$ satisfies
\[\langle\d^\star R(v_1,v_2),\,v_3\rangle=\langle(D_{v_3}\hbox{Ric}^\sharp)(v_1),\,v_2\rangle-\langle(D_{v_2}\hbox{Ric}^\sharp)(v_3),\,v_1\rangle.\]
Noting that $R(\nabla\phi)(u,v)=R(\nabla\phi, u)v$, to simplify notation, we also define
\begin{eqnarray}
	R^\sharp_\phi=\d^\star\!R+D{\rm Ric}^\sharp_\phi+R(\nabla\phi).
	\label{eqn12b}
\end{eqnarray}
 The bound on $Df_h$ requires restriction to the smaller and smoother class $\mathcal H_1$; the same is required when bounding $D^2f_h$.  Let
\begin{equation}
	\label{H_2}
	\mathcal H_2:=\{h\in\mathcal C_0(\M) \cap \mathcal C^2(\M) \mid C_0(h) < \infty, \thickspace C_1(h) < \infty,  \thickspace C_2(h) < \infty\}.
\end{equation}

\begin{proposition}
Assume that the conditions of Theorem $\ref{thm2}$ hold and that
	\[\chi^{\phantom{A}}_1=\sup_{x\in\M}\|R^\sharp_\phi\|_{op}(x)\quad\hbox{and}\quad\chi^{\phantom{A}}_2=m\sup_{x\in\M}\|R\|^2_{op}(x)\]
	are both finite, where $R^\sharp_\phi$ is defined by \eqref{eqn12b}. Further, assume that ${\rm Ric}^\sharp_\phi$, $R^\sharp_\phi$ and $R$ are all Lipschitz with finite Lipschitz constants $L({\rm Ric}^\sharp_\phi)$, $L(R^\sharp_\phi)$ and $L(R)$ respectively. For every $h \in \mathcal H_2$ with $f_h$ defined in \eqref{eqn10}: 
	\begin{enumerate}
		\item [(i)]
	If $\chi^{\phantom{A}}_2=0$, $D^2f_h$ exists and is Lipschitz with constant
	\[C_2(f_h)\leqslant\frac{1}{3\kappa}C_2(h)+\frac{3}{4\kappa^2}C_1(h)\,C_2(\phi)+C_0(h)\left(\frac{1}{4\kappa^2}C_3(\phi)+\frac{3}{4\kappa^3}C_2(\phi)^2\right).\]
	\item [(ii)]
If $\chi^{\phantom{A}}_2>0$ and $\kappa>1/2$, then $D^2f_h$ exists and is Lipschitz with constant
	\begin{eqnarray*}
		C_2(f_h)&\leqslant&
		C_2(h)\frac{1}{3\kappa}+C_1(h)\left\{\frac{L({\rm Ric}^\sharp_\phi)}{2\kappa^2}+\frac{4}{8\kappa-1}\left(\frac{\chi_1^2+2\chi^{\phantom{A}}_2}{4\kappa+1}\right)^{1/2}\right\}\\
		&&+\,\,C_0(h)\frac{2\tilde\beta}{2\kappa-1}
	\end{eqnarray*}
	where 
	\[\tilde\beta^2=\frac{\beta_1}{4\kappa+1}+\frac{\beta_2}{3\kappa+1}+\frac{\beta_3}{2\kappa+1},\]
	with
	\begin{eqnarray*}
		\beta_1&=&2mL(R)^2+\frac{1}{2}L(R^\sharp_\phi)^2+\frac{\chi_1^2+6\chi^{\phantom{A}}_2}{4\kappa+1}L({\rm Ric}^\sharp_\phi)^2,\\
		\beta_2&=&\frac{\chi^{\phantom{A}}_1}{\kappa}L({\rm Ric}^\sharp_\phi)\,L(R^\sharp_\phi),\\
		\beta_3&=&\frac{1}{\kappa^2}\left(\frac{\chi_1^2}{2}+2\chi^{\phantom{A}}_2\right)L({\rm Ric}^\sharp_\phi)^2.
	\end{eqnarray*}
		\end{enumerate}
\label{prop6}
\end{proposition}

\begin{remark}
	Note that, $\chi^{\phantom{A}}_2=0$ corresponds to $\M$ being a flat manifold, such as a Euclidean space, a cylinder or a flat torus. Consequently, $\chi^{\phantom{A}}_1=L({\rm Ric}^\sharp_\phi)=C_2(\phi)$ and $L(R^\sharp_\phi)=C_3(\phi)$. Our result thus recovers the corresponding bound given in \cite{MG} for $\mathbb R^m$, where $L_i$, $M_i(h)$ and $k$ in \cite{MG} correspond respectively to $C_{i-1}(\phi)$, $C_{i-1}(h)$ and $2\kappa$ here. Our result establishes that their upper bound also holds for general complete and connected flat manifolds. 
	
	On the other hand, if $\M$ is locally symmetric, we have $DR=0$. Then, it follows from \eqref{eqn12f} and \eqref{eqn12e} that $L({\rm Ric}^\sharp_\phi)=L(D_u(\nabla\phi))=C_2(\phi)$ and $L(R^\sharp_\phi)=L(D({\rm Ric}^\sharp_\phi))=C_3(\phi)$. As symmetric manifolds are locally symmetric, this will hold for a class of familiar manifolds, such as spheres, hyperbolic spaces, projective spaces and the space of positive definite symmetric matrices.
Pertinently, the upper bound for $C_2(f_h)$ in Proposition \ref{prop6} when $\chi_2=0$ is not the limit, as $\chi_2\rightarrow0$, of that for $\chi_2>0$. In addition, we need an extra requirement for $\kappa$ when $\chi_2>0$.
\end{remark}

\begin{proof}
	The proof uses Lemmas 4, 5, 6 and 7 given in Appendix B of Supplementary Material. 
	Consider the vector field $V^x_t$ along the path $X_{x,t}$ which satisfies the stochastic covariant It\^o equation 
	\begin{eqnarray}
		DV^x_t=R(\Xi\d B_t,u^x_t)v^x_t-\frac{1}{2}\left\{R^\sharp_\phi(u^x_t,v^x_t)+\hbox{Ric}^\sharp_\phi(V^x_t)\right\}\d t
		\label{eqn12c}
	\end{eqnarray}
	with $V^x_0=0$, where $\Xi$ is defined in \eqref{eqn1a}, $R^\sharp_\phi$ and $\hbox{Ric}^\sharp_\phi$ are defined by \eqref{eqn12b} and \eqref{eqn12d} respectively, and where $u^x_t$ and $v^x_t$ are the solutions of \eqref{eqn12a} both with the underlying path $X_{x,t}$ and with the initial conditions $u^x_0=u$ and $v^x_0=v$ respectively. It is known that, for $h$ satisfying the given conditions, $N'_s=D^2\E{h(X_{X_{x,s},t-s})}(u^x_s,v^x_s)+D\E{h(X_{X_{x,s},t-s})}(V^x_s)$ is a local martingale for $0\leqslant s\leqslant t$,  for any fixed $t>0$ \citep[Lemma 11.3]{JT}. Since
	\begin{eqnarray*}
		|N'_s|&\leqslant&
		\|\E{D^2h(X_{X_{x,s},t-s})}\|_{op}|u^x_s|\,|v^x_s|+\|\E{Dh(X_{X_{x,s},t-s})}\|_{op}\,|V^x_s|\\
		&\leqslant&C_1(h)|u^x_s|\,|v^x_s|+C_0(h)\,|V^x_s|,
	\end{eqnarray*}
	it follows from Lemmas 4 and 6 (Appendix B of Supplementary Material) that $\E{|N_s'|}<\infty$ so that $N'_s$ is in fact a martingale for $0\leqslant s\leqslant t$. Thus $\E{N'_0}=\E{N'_t}$, which implies that, for any fixed $t>0$ and $u,v\in T_x(\M)$, 
	\begin{eqnarray*}
		D^2\E{h(X_{x,t})}(u,v)=\E{D^2h(X_{x,t})(u^x_t,v^x_t)}+\E{D h(X_{x,t})(V^x_t)}.
	\end{eqnarray*}
	Then, the definition of $f_h$, the Dominated Convergence Theorem and Theorem \ref{thm4} together ensure that $D^2f_h$ exists and that, for any $u,v\in T_x(\M)$,
	\begin{eqnarray}
		\begin{array}{rcl}
			D^2f_h(x)(u,v)
			&=&\displaystyle\int_0^\infty\left\{\E{D^2h(X_{x,t})(u^x_t,v^x_t)}+\E{Dh(X_{x,t})(V^x_t)}\right\}\d t.
		\end{array}       
	\end{eqnarray}

	Now, we construct a pair of diffusions $(X_{x,t},Y_{y,t})$, starting from $(x,y)$, as in Theorem \ref{thm4}. Since we need to apply Lemmas 5 and 7 (Appendix B of Supplementary Material) to the processes related to $(X_{x,t},Y_{y,t})$ in the following proof, it is necessary to take the parameter $\ell$ in the construction of $(X_{x,t},Y_{y,t})$ to be 6. As in the proof of Proposition \ref{prop6a}, write $u^y_t$ and $v^y_t$ for the solutions of \eqref{eqn12a} with the underlying path $X_{x,t}$ replaced by $Y_{y,t}$ and with the respective initial conditions $u^y_0=\Pi_{\gamma_{x,y}}(u)$ and $v^y_0=\Pi_{\gamma_{x,y}}(v)$. Also, let $\tilde u^x_t$ denote $\Pi_{\gamma^{\vphantom{A}}_{X_{x,t},Y_{y,t}}}(u^x_t)$, and similarly for $\tilde v^y_t$ and $\tilde V^y_t$. Then,
	\begin{eqnarray}
		\begin{array}{rcl}
			&&|(D^2f_h(x)-\Pi_{\gamma_{x,y}}D^2f_h(y))(u,v)|\vspace{0.2cm}\\
			%&=&|D^2f_h(x)(u,v)-D^2f_h(y)(\Pi_{\gamma_{x,y}}u,\Pi_{\gamma_{x,y}}v)|\vspace{0.2cm}\\
			&\leqslant&\displaystyle\int_0^\infty\E{\left|D^2h(X_{x,t})(u^x_t,v^x_t)-D^2h(Y_{y,t})(u^y_t,v^y_t)\right|}\d t\\
			&&+\displaystyle\int_0^\infty\E{\left|Dh(X_{x,t})(V^x_t)-Dh(Y_{y,t})(V^y_t)\right|}\d t.
		\end{array}
		\label{eqn10c}
	\end{eqnarray}
	Under the given conditions on $h$, the first term on the right hand side of \eqref{eqn10c} can be estimated as 
	\begin{eqnarray*}
		&&\int_0^\infty\E{\left|D^2h(X_{x,t})(u^x_t,v^x_t)-D^2h(Y_{y,t})(u^y_t,v^y_t)\right|}\d t\\
		&\leqslant&\int_0^\infty\E{\left|\left(D^2h(X_{x,t})-\Pi_{\gamma^{\phantom{A}}_{X_{x,t},Y_{y,t}}}D^2h(Y_{y,t})\right)(u^x_t,v^x_t)\right|}\d t\\
		&&+\int_0^\infty\E{\left|D^2h(Y_{y,t})(\tilde u^x_t-u^y_t,\tilde v^x_t)\right|}\d t
		+\int_0^\infty\E{\left|D^2h(Y_{y,t})(u^y_t,\tilde v^x_t-v^y_t)\right|}\d t\\
		&\leqslant&C_2(h)\int_0^\infty\E{\rho(X_{x,t},Y_{y,t})\,|u^x_t|\,|v^x_t|}\,\d t
	  +C_1(h)\int_0^\infty\E{|\tilde u^x_t-u^y_t|\,|\tilde v^x_t|+|u^y_t|\,|\tilde v^x_t-v^y_t|}\d t.
	\end{eqnarray*}
	Similarly, for the second term on the right hand side of \eqref{eqn10c}, we have that
	\begin{eqnarray*}
		&&\int_0^\infty\E{\left|Dh(X_{x,t})(V^x_t)-Dh(Y_{y,t})(V^y_t)\right|}\d t\\
		&\leqslant&\int_0^\infty\E{\left|\left(Dh(X_{x,t})-\Pi_{\gamma^{\phantom{A}}_{X_{x,t},Y_{y,t}}}Dh(Y_{y,t})\right)(V^x_t)\right|}\d t
		+\int_0^\infty\E{\left|Dh(Y_{y,t})\left(\tilde V^x_t-V^y_t\right)\right|}\d t\\
		&\leqslant&C_1(h)\displaystyle\int_0^\infty\E{\rho(X_{x,t},Y_{y,t})\,|V^x_t|}\,\d t+C_0(h)\displaystyle\int_0^\infty\E{|\tilde V^x_t-V^y_t|}\d t.
	\end{eqnarray*}

	By the H\"older inequality, Theorem \ref{thm4} and Lemmas 4, 5, 6 and 7 (Appendix B of Supplementary Material), it follows from the above estimations and from \eqref{eqn10c} that, if $\chi^{\phantom{A}}_2>0$,
	\begin{eqnarray*}
		&&\frac{1}{\rho(x,y)\,|u|\,|v|}\left|(D^2f_h(x)-\Pi_{\gamma_{x,y}}D^2f_h(y))(u,v)\right|\\
		&\leqslant&
		C_2(h)\frac{1}{3\kappa}+C_1(h)\left\{\frac{L({\rm Ric}^\sharp_\phi)}{2\kappa^2}+\left(\frac{2\chi^{\phantom{A}}_2+\chi_1^2}{4\kappa+1}\right)^{1/2}\frac{4}{8\kappa-1}\right\}+C_0(h)\frac{2\tilde\beta}{2\kappa-1}
	\end{eqnarray*}
	when $\kappa>1/2$, as required. 
	
	If $\chi^{\phantom{A}}_2=0$, we need to modify the above application of Lemmas 6 and 7 (Appendix B of Supplementary Material). This results in 
	\begin{eqnarray*}
		&&\frac{1}{\rho(x,y)\,|u|\,|v|}\left|(D^2f_h(x)-\Pi_{\gamma_{x,y}}D^2f_h(y))(u,v)\right|\\
		&\leqslant&
		\frac{1}{3\kappa}C_2(h)+\frac{3}{4\kappa^2}C_1(h)\,C_2(\phi)+C_0(h)\left(\frac{1}{4\kappa^2}C_3(\phi)+\frac{3}{4\kappa^3}C_2(\phi)^2\right).
	\end{eqnarray*}
	This shows that $Ddf_h$ is Lipschitz with the required constant.
\end{proof}

\section{Application to bounding integral (semi-)metrics}
\label{bounds}
A key application of Stein's method is in obtaining upper bounds on an integral (semi-)metric $d_\mathcal{H}(X,Z)$, with respect to some function class $\mathcal{H}$, for an arbitrary random variable $Z\sim \nu$. Exploiting the characterising property of the operator $L_\phi$,
\[
\E{h(Z)}-\E{h(X)}=\E{L_\phi f_h(Z)}, \qquad \forall h \in \mathcal{H},
\]
the task then reduces to obtaining a uniform upper bound on $\E{L_\phi f_h(Z)}$ over functions $f_h$ using the Stein factors. The quantity $d_\mathcal H$ is clearly a semi-metric and is a metric only if $\mathcal H$ separates points in the set of signed measures on $\M$. 

\subsection{Wasserstein distance between $\mu_\phi$ and $\mu_\psi$}
\label{wasserstein_bound}

The result of Theorem \ref{thm2} in conjunction with the first-order bound in Proposition \ref{cor1} can be used to obtain an upper bound on the 1-Wasserstein distance between certain types of random variables. For this we consider the function class
\[\mathcal H^1_{\leq 1}:=\{h\in \mathcal C(\M) \mid h \text{ is Lipschitz with } C_0(h)\leqslant1\},\]
under which $d_{\mathcal H}$ is a bonafide metric. 
The 1-Wasserstein distance between two random variables $Z_1$ and $Z_2$ on $\M$ is then defined as
\[d_{\mathcal W}(Z_1,Z_2):=\sup_{h\in\mathcal H^1_{\leq 1}}|\E{h(Z_1)}-\E{h(Z_2)}|.\]

\begin{theorem}
	Assume that the conditions of Theorem $\ref{thm2}$ hold. Let $Z\sim\mu_\psi$ such that $\E{\rho(Z,x)}<\infty$ for some $x\in\M$, where $\psi$ satisfies \eqref{eqn2} with some constant $\kappa'>0$. Then
	\[d_{\mathcal W}(Z,X)\leqslant\frac{1}{2\kappa}\E{\left|\nabla(\psi-\phi)(Z)\right|}.\]
	\label{thm5}
\end{theorem}
\begin{proof}
The proof pursues a similar argument to that of Proposition 4.1 of \cite{MRS}.  Note first that
	\[\sup_{h\in\mathcal H^1_{\leq 1}}|\E{h(Z)}-\E{h(X)}|=\sup_{h\in\mathcal H^1_{\leq 1}\cap\mathcal C_0(\M)}|\E{h(Z)}-\E{h(X)}|.\]
	For $h\in\mathcal H^1_{\leq 1}\cap\mathcal C_0(\M)$, we have by Theorem \ref{thm2} that
	\[\E{h(Z)}-\E{h(X)}=\E{L_\phi f_h(Z)}.\]
	On the other hand, the given assumption that $Z\sim\mu_\psi$, where $\psi$ satisfies \eqref{eqn2}, also implies that $\E{L_\psi f_h(Z)}=0$ for $h\in\mathcal H^1_{\leq 1}\cap\mathcal C_0(\M)$. Noting that
	\[L_\phi f_h(x)=L_\psi f_h(x)+\frac{1}{2}\langle\nabla\psi(x)-\nabla\phi(x),f_h(x)\rangle,\]
	we obtain
	\[\E{h(Z)}-\E{h(X)}=\frac{1}{2}\E{\langle\nabla\psi(Z)-\nabla\phi(Z),\nabla f_h(Z)\rangle},\]
	so that the result follows from Proposition \ref{cor1}.
\end{proof}

\begin{example}
	Assume that $\M=S^m$ and that all probability measures $\mu_\varphi$ involved satisfy the condition 
	\[\hbox{{\rm Hess}}^\varphi\geqslant(2\kappa-(m-1))\,g,\] 
	for some $\kappa>0$. 
	
	\begin{enumerate}[label=(\roman*), leftmargin=*]
		\itemsep 1.3em
		\item The functions $\phi$ and $\psi$ corresponding to von Mises-Fisher distributions $M(x_1,c_1)$ and $M(x_2,c_2)$ are respectively $-c_1\cos\rho(x_1,x)$ and $-c_2\cos\rho(x_2,x)$. Then,
		\[\left|\nabla(\psi-\phi)(x)\right|=c^*|\sin\rho(x^*,x)|\leqslant c^*\rho(x^*,x)\leqslant c^*\left\{\rho(x^*,x_2)+\rho(x_2,x)\right\},\]
		where $c^*=|c_2x_2-c_1x_1|$ and $x^*=(c_2x_2-c_1x_1)/c^*$. From the symmetry between $\phi$ and $\psi$, it follows that the Wasserstein-1 distance $d_{\mathcal W}$ between $M(x_1,c_1)$ and $M(x_2,c_2)$ is bounded:
		\[d_{\mathcal W}(X_1,X_2)\leqslant\frac{|c_2x_2-c_1x_1|}{4\kappa}\left\{\sum_{i=1}^2(\rho(x^*,x_i)+\E{\rho(x_i,X_i)})\right\},\]
		where $X_i\sim M(c_i,x_i)$.
		
		\item  The function $\psi$ corresponding to the Fisher-Watson distribution 
		\[W(x_1,x_2,c_1,c_2)\propto e^{c_1\langle x_1,x\rangle+c_2\langle x_2,x\rangle^2}\d\vol(x),\]
		where $\langle x_1,x_2\rangle=0$, is $-c_1\cos\rho(x_1,x)-c_2\cos^2\rho(x_2,x)$. If $\mu_\phi$ is the von Mises-Fisher distribution $M(x_1,c_1)$, then
		\[\left|\nabla(\psi-\phi)(x)\right|=c_2|\sin(2\rho(x_2,x))|.\]
		Hence, for $X\sim  M(x_1,c_1)$ and $Z\sim W(x_1,x_2,c_1,c_2)$,
		\[d_{\mathcal W}(X,Z)\leqslant\frac{c_2}{2\kappa}\E{|\sin(2\rho(x_2,Z))|}.\]

\label{ex6}
\item 
	Let $m>2$ and $\M=SO(m)$ with the bi-invariant metric determined by $g(E_1,E_2)\!=\!-\frac{1}{2}\text{tr}(E_1E_2)$ for skew-symmetric $E_1,E_2$. Assume that, for $S\in\M$, $\phi(S)=-c\,\text{tr}(S_0S)$ with $S_0\in SO(m)$ and that constant $c>0$. Then, $\mu_\phi$ is a von Mises-Fisher distribution on $SO(m)$. Since for any skew-symmetric matrix $E$ 
	\[\lim_{t\rightarrow0}\frac{\phi(S\,e^{tE})-\phi(S)}{t}=-c\,\text{tr}(S_0SE),\]
	we have that $\nabla\phi(S)=\frac{c}{\sqrt2}S\{(S_0S)^\top-S_0S\}$. This implies that
	\begin{eqnarray*}
		\frac{2}{c^2}|\nabla\phi(S)|^2&=&\langle S\{(S_0S)^\top-S_0S\},\,S((S_0S)^\top-S_0S)\rangle_S\\
		&=&\langle(S_0S)^\top-S_0S,\,(S_0S)^\top-S_0S\rangle_I\\
		&=&-\text{tr}\left(((S_0S)^\top-S_0S)^2\right)=2(m-\text{tr}((S_0S)^2)).
	\end{eqnarray*}

	If $c\in(0,(m-2)/2)$, there is a $\kappa>0$ such that the Bakry-Emery curvature criterion (A1) holds, as seen in Example \ref{examples}(iv). Then, if $Z$ is a uniform random variable on $SO(m)$, $S_0Z$ is also a uniform random variable and so
	\[d_{\mathcal W}(Z,X)\leqslant\frac{c}{2\kappa}\E{\sqrt{m-\text{tr}(Z^2)}}.\]
	\label{ex7}
	\end{enumerate}
\end{example}

\subsection{Integral semi-metrics for general distributions}
\label{general_bound}
If $h \in \mathcal H_2$, the result of Theorem \ref{thm2}, together with Propositions \ref{prop6a} and \ref{prop6}, enable us to bound $\E{h(Z)}-\E{h(X)}$ for a more general random variable $Z$ on $\M$ as follows, where $\mathcal H_2$ is as defined in \eqref{H_2}.

\begin{corollary}
Assume that the conditions of Proposition $\ref{prop6}$ hold. Assume further that $\phi$ is Lipschitz with Lipschitz constants $C_i(\phi)$, $i=0,1$. Then for every $h \in \mathcal H_2$
\[|\E{h(Z)}-\E{h(X)}|\leqslant\eta\E{\rho(Z,X)},\]
where 
\[\eta=mC_2(f_h)+C_0(\phi)C_1(f_h)+C_1(\phi)C_0(f_h)\]
and where $C_i(f_h)$ are bounded as in Propositions $\ref{prop6a}$ and $\ref{prop6}$.
\label{cor4}
\end{corollary}

\begin{proof}
It follows from a direct estimation of $|\E{L_\phi(f_h)(Z)}|$ that
\begin{eqnarray*}
|\E{L_\phi(f_h)(Z)}|
&=&|\E{\left(L_\phi(f_h)(Z)-L_\phi(f_h)(X)\right)}|\\
&\leqslant&|\E{\left(\Delta(f_h)(Z)-\Delta(f_h)(X)\right)}|\\
&&+|\E{\langle\nabla\phi(Z),\nabla f_h(Z)\rangle-\langle\nabla\phi(X),\nabla f_h(X)\rangle}|\\
&\leqslant&mC_2(f_h)\E{\rho(Z,X)}\\
&&+\,\left|\E{\langle\nabla\phi(Z),\nabla f_h(Z)\rangle-\langle\Pi_{\gamma^{\phantom{A}}_{X,Z}}\nabla\phi(X),\nabla f_h(Z)\rangle}\right|\\
&&+\,\left|\E{\langle\nabla\phi(X),\Pi_{\gamma^{\phantom{A}}_{Z,X}}\nabla f_h(Z)\rangle-\langle\nabla\phi(X),\nabla f_h(X)\rangle}\right|\\
&\leqslant&\left\{mC_2(f_h)+C_0(f_h)C_1(\phi)+C_1(f_h)C_0(\phi)\right\}\E{\rho(Z,X)}
\end{eqnarray*}
as required.
\end{proof}

A further simplification occurs when $\M$ is compact. 

\begin{corollary}
If $\M$ is compact then, for any Lipschitz function on $\M$ with $C_0(h)\leqslant1$, any fixed $\epsilon>0$ and $s>0$, there exists a $g \in \mathcal{C}^2(\M)$ with Lipschitz constants $C_i(g), i=0,1,2$, such that $C_0(g)\leqslant 1+s$ and 
\[|\E{h(Z)}-\E{h(X)}|\leqslant 2\epsilon+\left\{mC_2(f_g)+C_0(f_g)C_1(\phi)+C_1(f_g)C_0(\phi)\right\}\E{\rho(Z,X)}.\]
\end{corollary}

\begin{proof}
Since $\M$ is compact, any $g\in\mathcal C^\infty(\M)$ has bounded derivatives, and thus possesses finite Lipschitz constant $C_i(g), i=0,1,2,\ldots, k$ for every $k$. This ensures that Lipschitz constants $C_i(f_g), i=0,1,2$ of the Stein equation solution $f_g$ are finite. 

The existence of the requisite $g \in \mathcal{C}^2(\M)$ is guaranteed by the result in \cite{AFLR} on existence of a $\mathcal{C}^\infty$ Lipschitz approximation of a Lipschitz function. By Theorem 1 in \cite{AFLR}, for every Lipschitz function $h$ on $\M$ with Lipschitz constant 1 and for every $\epsilon, s>0$, there exists a $g \in \mathcal{C}^\infty(\M)$ such that
$\sup_{x \in \M}|g(x)-h(x)|<\epsilon \text{ with } C_0(g)\leqslant 1+s$. Thus, by applying Corollary \ref{cor4} to $g$, we have
\begin{eqnarray*}
&&|\E{h(Z)}-\E{h(X)}|\\
&\leqslant&|\E{h(Z)}-\E{g(Z)}|+|\E{g(X})-\E{h(X)}|+|\E{g(Z)}-\E{g(X)}| \\
&\leqslant&2\epsilon
 + |\E{g(Z)}-\E{g(X)}| \\
& \leqslant& 2\epsilon+\left\{mC_2(f_g)+C_0(f_g)C_1(\phi)+C_1(f_g)C_0(\phi)\right\}\E{\rho(Z,X)},
\end{eqnarray*}
as required. 
\end{proof}
 
Consider the function class
\[
\mathcal H^2_{\leq 1}=\{h\in\mathcal C^2(\M)\mid h\hbox{ is Lipschitz with }C_0(h) \leqslant1, C_1(h) \leqslant1,C_2(h) \leqslant1\}.
\] 
Since 
\[\sup_{h\in\mathcal H^2_{\leq 1}}\left|\E{h(Z)}-\E{h(X)}\right|=\sup_{h\in\mathcal H^2_{\leq 1}\cap\mathcal C_0(\M)}\left|\E{h(Z)}-\E{h(X)}\right|,\]
from Propositions $\ref{prop5}$, $\ref{prop6a}$ and $\ref{prop6}$, as well as Corollary \ref{cor4}, the following result on the bound for the integral (semi-)metric 
\[
d_{I}(Z,X):= \sup_{h \in \mathcal H^2_{\leq 1}} \left|E[h(Z)]-E[h(X)]\right|,
\]
is immediate.

\begin{theorem}
Assume that the conditions of Proposition $\ref{prop6}$ hold, and that $\phi$ is Lipschitz with Lipschitz constants $C_i(\phi)$, $i=0,1$. Then, for any random variable $Z$ on $\M$,
\[d_I(Z,X)\leqslant\eta^*\E{\rho(Z,X)},\]
where, if $\chi^{\phantom{A}}_2=0$,
\begin{eqnarray*}
\eta^*&=&
m\left\{\frac{1}{3\kappa}+\frac{1}{4\kappa^2}\left(3C_2(\phi)+C_3(\phi)\right)+\frac{3}{4\kappa^3}C_2(\phi)^2\right\}\\
&&+C_0(\phi)\left\{\frac{1}{2\kappa}+\frac{C_2(\phi)}{2\kappa^2}\right\}+C_1(\phi)\frac{1}{\kappa}
\end{eqnarray*}
while, if $\chi^{\phantom{A}}_2>0$, 
\begin{eqnarray*}
\eta^*&=&
m\left\{\frac{1}{3\kappa}+\frac{L({\rm Ric}^\sharp_\phi)}{2\kappa^2}+\frac{4}{8\kappa-1}\left(\frac{2\chi^{\phantom{A}}_2+\chi_1^2}{4\kappa+1}l\right)^{1/2}+\frac{2\tilde\beta}{2\kappa-1}\right\}\\
&&+C_0(\phi)\left\{\frac{1}{2\kappa}+\frac{L({\rm Ric}^\sharp_\phi)}{2\kappa^2}\right\}+C_1(\phi)\frac{1}{\kappa},
\end{eqnarray*}
and where the constants $\chi^{\phantom{A}}_1$, $\chi^{\phantom{A}}_2$ and $\tilde\beta$ are as in Proposition $\ref{prop6}$.
\label{thm6}
\end{theorem}

%\begin{supplement}
%\stitle{Auxiliary results}
%\sdescription{Appendix A contains a result on behaviour of the distance function around the first non-conjugate point of the cut locus, which is used in the proofs in Section 3. Appendix B is devoted to estimation of certain stochastic vector fields used in the proofs in Section 4.2.}
%\end{supplement}

%%%%%%%%%%%%%%%%%%%%%%%%%%%%%%%%%%%%%%%%%%%
\section*{Appendix}
\subsection{Appendix A\label{A1}}
The following is unpublished work of Barden \& Le (1997). It concerns the behaviour of the distance function $\rho$ on $\M\times\M$ in a neighbourhood of $(x,y)$ when $y$ is a non-first-conjugate cut point of $x$. It generalises a result of \cite{BL1}, where similar behaviour was studied when one of the two points in $\M$ is fixed, and is used in proofs of Theorems 1 and 2. 

\vskip 6pt
\setcounter{lemma}{2}
\begin{lemma}
	Given a point $x$ of a complete and connected Riemannian manifold $\M$ and a point $y$ in the non-conjugate part of the cut locus of $x$ $($and hence also vice-versa$)$, there is a neighbourhood $\mathcal N$ of $(x,y)$ in $\M\times\M$, a maximal integer $s>1$, and finitely many smooth functions $\varrho_i$, $i=1,\cdots,s$, on $\mathcal N$ such that $\varrho_i(x,y)=\rho(x,y)$ for $i=1,\cdots,s$; the germs of $\varrho_i$ at $(x,y)$ are uniquely defined, up to their ordering, and, for all $(x',y')$ in $\mathcal N$,
	\[\rho(x',y')=\min_{1\le i\le s}\left\{\varrho_i(x',y')\right\}.\]
	The set
	\[\mathcal N_{ij}=\{(x,y)\in\M\times\M\mid\varrho_i(x,y)=\varrho_j(x,y)\}\]
	is a co-dimension 1 submanifold of $\mathcal N$ meeting each slice $\{x\}\times\M$ transversely, and the set of points of $\mathcal N$ at which three or more of the $\varrho_i$ are equal is a finite union of co-dimension $2$ submanifolds, meeting $\{x\}\times\M$ transversely.
	\label{lem3}
\end{lemma}

The germ of a function at a point $x$ is its equivalence class under agreement on a (variable) neighbourhood of $x$.

\begin{proof} Recall the result of Proposition 2 in \cite{BL1}: for a fixed point $x\in\M$, its (first) conjugate locus in its cut locus has co-dimension at least 2; and for $y$ in the non-conjugate part of the cut locus of $x$ there are finitely many locally defined smooth functions each of whose value at $y$ is $\rho(x,y)$, whose germ at $y$ is well-defined and such that the minimum of these functions on a neighbourhood of $y$ is the distance from $x$. These functions are induced by the radial distance in $T_x(\M)$ and were termed `radial functions'. They arise as follows. We let $v_1,\cdots,v_s\in T_x(\M)$ determine the full set of minimal geodesics $\gamma_i(t)=\exp_x(t v_i)$, $0\leqslant t\leqslant1$, from $x$ to $y$. The number $s$ of such geodesics will vary with $y$ but, outside the conjugate locus, is necessarily finite. We may choose neighbourhoods $\tilde{\mathcal V}_i$ of $v_i$ in $T_x(\M)$ and $\mathcal V$ of $y$ in $\M$ with each $\tilde{\mathcal V}_i$ mapped diffeomorphically onto $\mathcal V$ by $\exp_x$. In particular, $\mathcal V$ will be sufficiently small not to meet the conjugate locus. Then the relevant smooth functions, which were denoted by $\phi_i$ in \cite{BL1}, are given by
	\[\phi_i:\,\,V\longrightarrow{\mathbb R}_+;\,\,y\mapsto\left\Vert\left(\exp_x\big|_{\tilde{\mathcal V}_i}\right)^{-1}(y)\right\Vert.\]

	Now we consider the `full' exponential map Exp which maps $v\in T_x(\M)$ to $(x,\,\exp_x(v))$. If $\pi:\,\,\M\times\M\longrightarrow\M$ projects onto the second factor, then $\pi\circ\text{Exp}\big|_{T_x(\M)}=\exp_x$. Thus Exp is a differentiable mapping and, with respect to local coordinates on $\TM$ which are the product of those at $x$ on $\M$ and those in $T_x(\M)$, at any point $v$ in $T_x(\M)$, its derivative takes the form
	\[\left(\begin{matrix}Id&*\\0&D\exp_x(v)\end{matrix}\right).\]
	Thus Exp is non-singular at $v$ if and only if $\exp_x$ is non-singular at $v$.
	
	Then each $v_i$ has a neighbourhood $\tilde{\mathcal N}_i$ in $\TM$ such that $\tilde{\mathcal N}_i\cap T_x(\M)\subset\tilde{\mathcal V}_i$ and each of the restrictions $\text{Exp}|_{\tilde{\mathcal N}_i}$ is a diffeomorphism onto the same neighbourhood $\mathcal N$ of $(x,y)$ in $\M\times\M$. We may also choose $\mathcal N$ sufficiently small that for all $(x',y')\in\mathcal N$ any minimal geodesic from $x'$ to $y'$ is determined by $v'$ in some $\tilde{\mathcal N}_i\cap T_{x'}(\M)$. For otherwise there exists an infinite sequence of points $(x_n,y_n)$ in $\M\times\M$ converging to $(x,y)$ with minimal geodesics $\alpha_n(t)=\exp_{x_n}(tu_n)$, $0\leqslant t\leqslant1$, from $x_n$ to $y_n$ where $u_n$, the initial tangent vector to $\alpha_n$, is in $T_{x_n}(\M)$ but in no $\tilde{\mathcal N}_i$. Then, since $\alpha_n$ is minimal, the distance $\Vert u_n\Vert$ from $x_n$ to $y_n$ along $\alpha_n$ is equal to $\rho(x_n,y_n)$, which converges to $\rho(x,y)$. So the sequence $\{u_n\}$ is bounded and, without loss of generality, converges to $u$ say. However $\alpha_n(1)=y_n$ for all $n$ so $y=\lim\limits_{n\rightarrow\infty}\alpha_n(1)=\exp_x(u)$ and as $\Vert u \Vert=\rho(x,y)$, $u$ must be one of the $v_i$. Hence the $\{u_n\}$ would eventually lie in the corresponding $\tilde{\mathcal N}_i$ contrary to our hypothesis on the $u_n$. Thus, having chosen $\mathcal N$ sufficiently small to avoid the finitely many such points $(x_n,y_n)$ it follows that, for all $(x',y')$ in $\mathcal U$ each minimal geodesic $\gamma'(t)=\exp_{x'}(tu')$, $0\leqslant t\leqslant1$, from $x'$ to $y'$ has $u'$ in one of the sets $\tilde{\mathcal N}_i\cap T_{x'}(\M)$. Then $\rho(x',y')=\Vert u'\Vert=\varrho_i(x',y')$, where
	\[\varrho_i:\,\,\mathcal N\longrightarrow{\mathbb R}_+;\,\,(x',y')\mapsto\left\Vert\left(\text{Exp}\big|_{\tilde{\mathcal N}_i}\right)^{-1}(x',y')\right\Vert\]
	since ${\rm Exp}(x',u')=(x',y')$. Hence, for all $(x',y')$ in $\mathcal N$, $\rho(x',y')=\min\limits_{1\le i\le s}\varrho_i(x',y')$ as required.
	
	The remaining claims follow from those for the restrictions $\phi_i$ of $\varrho_i$ to $\mathcal V_x=(\{x\}\times\M)\cap\mathcal N$ established in \cite{BL1}: since $\nabla_y(\varrho_i-\varrho_j)\not=0$, we certainly have $\nabla(\varrho_i-\varrho_j)(x,y)\not=0$ making $\mathcal N_{ij}$ a co-dimension 1 manifold and, since it meets each $\mathcal V_x$ in a co-dimension 1 manifold $\mathcal V_{ij}$, it must do so transversally. Similarly, it is not possible for $\nabla(\varrho_i-\varrho_j)$ and $\nabla(\varrho_i-\varrho_k)$ to be linearly dependent at points of $\mathcal N_{ij}\cap\mathcal N_{ik}$ since their components in $\mathcal V_{ij}\cap\mathcal V_{ik}$ are not. Thus, $\mathcal N_{ij}$ and $\mathcal N_{ik}$ meet transversally in a co-dimension 2 submanifold, transverse to each $\mathcal V_x$.
\end{proof}

\vskip 6pt 
We note that, although all $s$ functions $\varrho_i$ take the value $\rho(x,y)$ at $(x,y)$, as $(x',y')$ varies in $\mathcal N$ certain of the $\varrho_i$ will take values greater than $\rho(x',y')$. In particular for points $(x',y')$ where $y'$ is not in the cut locus of $x'$, precisely one will take this value. When $y'$ is in the cut locus of $x'$ then at least two, but not necessary all, of the functions $\varrho_i$ take the value $\rho(x',y')$.
%%%%%%%%%%%%%%%%%%%%%%%%%%%%%%%%%%%%%%%%%%%
\subsection{Appendix B}
\label{A2}
This Appendix investigates certain properties of the stochastic vector fields 
\begin{eqnarray}
		\frac{Dv^x_t}{\d t}=-\frac{1}{2}\text{Ric}^\sharp_\phi(v^x_t)
		\label{eqn12a}
	\end{eqnarray}
and 
\begin{eqnarray}
		DV^x_t=R(\Xi\d B_t,u^x_t)v^x_t-\frac{1}{2}\left\{R^\sharp_\phi(u^x_t,v^x_t)+\hbox{Ric}^\sharp_\phi(V^x_t)\right\}\d t\thickspace,
		\label{eqn12c}
	\end{eqnarray}
where $\Xi$ is as defined in equations (4) and (7) in the main paper, and
\begin{eqnarray}
	R^\sharp_\phi=\d^\star\!R+D{\rm Ric}^\sharp_\phi+R(\nabla\phi).
	\label{eqn12b}
\end{eqnarray}
These properties play an important role in our study of the derivatives of $f_h$ in Section 5.  One of the main tools which will repeatedly be used here is the H\"older inequality.    

\vskip 6pt
First, we have the following bound related to the  $v^x_t$ defined by \eqref{eqn12a}.

\begin{lemma}
	Assume that $\M$ is a complete and connected Riemannian manifold and that the Bakry-Emery curvature criterion (A1) is satisfied for a constant $\kappa>0$. Then, the vector field $v^x_t$ defined by \eqref{eqn12a} satisfies
	\[\E{|v^x_t|^q}\leqslant|v|^qe^{-q\kappa t}\]
	for any $q>1$.
	\label{lem4}
\end{lemma}

The special case for $q=2$ of this result was given in \cite{AT}.%, (7.4).

\begin{proof}
	\begin{eqnarray*}
		\d|v^x_t|^2=2\left\langle v^x_t,\,-\frac{1}{2}\hbox{Ric}^\sharp_\phi(v^x_t)\right\rangle\,\d t
		=-\left\{\hbox{Hess}^\phi(v^x_t,v^x_t)+\hbox{Ric}(v^x_t,v^x_t)\right\}\d t.
	\end{eqnarray*}
	Regarding $|v^x_t|^q$ as a function of $|v^x_t|^2$, it follows that, for $q>1$,
	\[\E{|v^x_t|^q}=|v|^q-\frac{q}{2}\int_0^t\E{(\hbox{Ric}(v^x_s,v^x_s)+\hbox{Hess}^\phi(v^x_s,v^x_s))\,|v^x_s|^{q-2}}\d s.\]
	Thus, using the Bakry-Emery curvature criterion (A1), we have that
	\[\d\E{|v^x_t|^q}\leqslant-q\kappa\E{|v^x_t|^q}\d t,\]
	so that the required inequality follows.
\end{proof}

Now, for $x,y\in\M$ and $v\in T_x(\M)$, let $v_t^x$ be as given in \eqref{eqn12a} and let $v^y_t$ be the solution to \eqref{eqn12a} with the underlying path $X_{x,t}$ replaced by $Y_{y,t}$ and with the initial condition $v^y_0=\Pi_{\gamma_{x,y}}(v)$, where $Y_{y,t}$ is constructed as in the proof of Theorem 2 with $\ell=2q$ for each fixed $q\geqslant1$ specified in the following Lemma and where, as earlier,  $\gamma_{x,y}$ denotes a minimal geodesic from $y$ to $x$ and $\Pi_{\gamma_{x,y}}$ denotes the parallel transport from $T_y(\M)$ to $T_x(\M)$ along $\gamma_{x,y}$. 

\begin{lemma}
	Assume that $\M$ is a complete and connected Riemannian manifold and that the Bakry-Emery curvature criterion (A1) is satisfied for a constant $\kappa>0$. If ${\rm Ric}^\sharp_\phi$ is Lipschitz with Lipschitz constant $L({\rm Ric}^\sharp_\phi)$ then, for $q\geqslant1$,
	\[\E{\left|\Pi_{\gamma^{\phantom{A}}_{X_{x,t},Y_{y,t}}}(v^x_t)-v^y_t\right|^q}\leqslant\left(\frac{L({\rm Ric}^\sharp_\phi)}{2\kappa}\right)^q\rho(x,y)^q|v|^qe^{-q\kappa t}.\]
	\label{lem4a}
\end{lemma}

\begin{proof}
	As in the proof of Proposition 4, for any $u_t\in T_{X_{x,t}}(\M)$, let $\tilde u_t$ denote $\Pi_{\gamma^{\phantom{A}}_{X_{x,t},Y_{y,t}}}(u_t)$. Then, from \eqref{eqn12a}, it follows that
	\begin{eqnarray*}
		\d\,|\tilde v^x_t-v^y_t|^2%\\
		%&=&\left\langle\tilde v^x_t-v^y_t,\,\text{Ric}^\sharp_\phi(Y_{y,t})(v^y_t)-\Pi_{\gamma^{\phantom{A}}_{X_{x,t},Y_{y,t}}}\left(\text{Ric}^\sharp_\phi(X_{x,t})(v^x_t)\right)\right\rangle\,\d t\\
		&=&-\left\langle\tilde v^x_t-v^y_t,\,\text{Ric}^\sharp_\phi(Y_{y,t})(\tilde v^x_t))-\text{Ric}^\sharp_\phi(Y_{y,t})(v^y_t)\right\rangle\,\d t\\ 
		&&-\left\langle\tilde v^x_t-v^y_t,\,\Pi_{\gamma^{\phantom{A}}_{X_{x,t},Y_{y,t}}}\left(\text{Ric}^\sharp_\phi(X_{x,t})(v^x_t)\right)-\text{Ric}^\sharp_\phi(Y_{y,t})(\tilde v^x_t)\right\rangle\,\d t.
	\end{eqnarray*}
	By the Bakry-Emery curvature criterion (A1) and the assumption on $\hbox{Ric}^\sharp_\phi$, this gives that
	\begin{eqnarray*}
		\d\,\left|\tilde v^x_t-v^y_t\right|^2%\\  
		&\leqslant&-\left(\text{Ric}+\text{Hess}^\phi\right)(\tilde v^x_t-v^y_t,\,\tilde v^x_t-v^y_t)\,\d t\\ 
		&&+|\tilde v^x_t-v^y_t|\,\left|\text{Ric}^\sharp_\phi(X_{x,t})(v^x_t)-\Pi^{-1}_{\gamma^{\phantom{A}}_{X_{x,t},Y_{y,t}}}\left(\text{Ric}^\sharp_\phi(Y_{y,t})\right)(v^x_t)\right|\,d t\\
		&\leqslant&-2\kappa\,|\tilde v^x_t-v^y_t|^2\,\d t+L({\rm Ric}^\sharp_\phi)\,|\tilde v^x_t-v^y_t|\,|v^x_t|\,\rho(X_{x,t},\,Y_{y,t})\,\d t.
	\end{eqnarray*} 
	For $q\geqslant1$, treating $\left|\tilde v^x_t-v^y_t\right|^q$ as a function of $\left|\tilde v^x_t-v^y_t\right|^2$, we have that
	\begin{eqnarray}
		%\begin{array}{rcl}
		\d\,\left|\tilde v^x_t-v^y_t\right|^q\vspace{0.1cm}%\\
		%&=&\dfrac{q}{2}\left|\tilde v^x_t-v^y_t\right|^{q-2}\d\,\left|\tilde v^x_t-v^y_t\right|^2\\ 
		\leqslant-\kappa\,q\,|\tilde v^x_t-v^y_t|^q\,\d t+\dfrac{q}{2}\,L({\rm Ric}^\sharp_\phi)\,|\tilde v^x_t-v^y_t|^{q-1}|v^x_t|\,\rho(X_{x,t},\,Y_{y,t})\,\d t.
		%\end{array}
		\label{eqn11a}
	\end{eqnarray} 

	We now consider the cases $q=1$ and $q>1$ separately.
	
	\vskip 6pt
	$(i)$ If $q=1$, \eqref{eqn11a} simplifies to
	\begin{eqnarray*}
		\d\,\left|\tilde v^x_t-v^y_t\right|\leqslant-\kappa\,|\tilde v^x_t-v^y_t|\,\d t+\frac{1}{2}\,L({\rm Ric}^\sharp_\phi)\,|v^x_t|\,\rho(X_{x,t},\,Y_{y,t})\,\d t.
	\end{eqnarray*}
	Thus, after taking expectations on both sides, the H\"older inequality, Theorem 2 and Lemma \ref{lem4} together give that
	\begin{eqnarray*}
		\d\,\E{\left|\tilde v^x_t-v^y_t\right|}%&\leqslant&-\kappa\,\E{|\tilde v^x_t-v^y_t|}\,\d t\\
		%&&+\frac{1}{2}\,L({\rm Ric}^\sharp_\phi)\E{|v^x_t|^2}^{1/2}\E{\rho(X_{x,t},\,Y_{y,t})^2}^{1/2}\,\d t\\
		&\leqslant&-\kappa\E{|\tilde v^x_t-v^y_t|}\,\d t+\frac{1}{2}\,L({\rm Ric}^\sharp_\phi)\,\rho(x,y)\,|v|\,e^{-2\kappa t}\d t.
	\end{eqnarray*}
	Hence,
	\[\d\left(e^{\kappa t}\E{|\tilde v^x_t-v^y_t|}\right)\leqslant\frac{1}{2}\,L({\rm Ric}^\sharp_\phi)\,\rho(x,y)\,|v|\,e^{-\kappa t}\d t,\]
	so that
	\begin{eqnarray*}
		\E{|\tilde v^x_t-v^y_t|}%&\leqslant&\frac{1}{2}\,L({\rm Ric}^\sharp_\phi)\,\rho(x,y)\,|v|\,e^{-\kappa t}\int_0^te^{-\kappa s}\d s\\
		&\leqslant&\frac{1}{2\kappa}\,L({\rm Ric}^\sharp_\phi)\,\rho(x,y)\,|v|\,e^{-\kappa t}
	\end{eqnarray*}
	as required.
	
	\vskip 6pt
	$(ii)$ For $q>1$, we use similar arguments to those in $(i)$ above. First, applying the H\"older inequality, Theorem 2 and Lemma \ref{lem4}, we obtain from \eqref{eqn11a} that
	\begin{eqnarray*}
		\d\E{\left|\tilde v^x_t-v^y_t\right|^q}%\\
		%&\leqslant&-\kappa\,q\,\E{|\tilde v^x_t-v^y_t|^q}\,\d t+\frac{q}{2}\,L({\rm Ric}^\sharp_\phi)\,\E{|\tilde v^x_t-v^y_t|^{q-1}|v^x_t|\,\rho(X_{x,t},\,Y_{y,t})}\,\d t\\
		&\leqslant&-\,\kappa\,q\,\E{|\tilde v^x_t-v^y_t|^q}\,\d t\\
		&&\!\!\!+\,\frac{q\,L({\rm Ric}^\sharp_\phi)}{2}\,\E{|\tilde v^x_t-v^y_t|^q}^{1-1/q}\E{|v^x_t|^{2q}}^{1/(2q)}\E{\rho(X_{x,t},\,Y_{y,t})^{2q}}^{1/(2q)}\d t\\
		&\leqslant&-\,\kappa\,q\,\E{|\tilde v^x_t-v^y_t|^q}\,\d t+\frac{q}{2}\,L({\rm Ric}^\sharp_\phi)\,\rho(x,y)\,|v|\E{|\tilde v^x_t-v^y_t|^q}^{1-1/q}e^{-2\kappa t}\d t.
	\end{eqnarray*} 
	Then, writing $y_t=\left\{\E{|\tilde v^x_t-v^y_t|^q}\right\}^{1/q}$, the above gives that
	\[\d y_t\leqslant-\kappa\,y_t\,\d t+\frac{1}{2}\,L({\rm Ric}^\sharp_\phi)\,\rho(x,y)\,|v|\,e^{-2\kappa t}\d t.\]
	As in $(i)$, it follows that 
	%
	%\[\d\left(e^{\kappa t}y_t\right)\leqslant\frac{1}{2}\,L({\rm Ric}^\sharp_\phi)\,\rho(x,y)\,|v|\,e^{-\kappa t}\d t,\]
	%
	%so that
	%
	\[y_t\leqslant\frac{1}{2\kappa}\,L({\rm Ric}^\sharp_\phi)\,\rho(x,y)\,|v|\,e^{-\kappa t},\]
	as required.
\end{proof}

We now turn to the vector field $V^x_t$ defined by \eqref{eqn12c},  recalling that, if $\chi^{\phantom{A}}_2=0$, then $\chi^{\phantom{A}}_1=C_2(\phi)$.

\begin{lemma}
	Assume that $\M$ is a complete and connected Riemannian manifold of dimension $m$ and that the Bakry-Emery curvature criterion (A1) is satisfied for a constant $\kappa>0$. Assume further that
	\[\chi^{\phantom{A}}_1=\sup_{x\in\M}\|R^\sharp_\phi\|_{op}(x)\quad\hbox{and}\quad\chi^{\phantom{A}}_2=m\sup_{x\in\M}\|R\|^2_{op}(x)\]
	are both finite, where $R^\sharp_\phi$ is defined by \eqref{eqn12b}. Then, for $q\geqslant2$, 
	\begin{eqnarray*}
		\E{|V^x_t|^q}\leqslant\left\{
		\begin{array}{ll}
			\left(\dfrac{C_1(\phi)}{2\kappa}\right)^q|u|^q|v|^qe^{-q\kappa t}&\hbox{if }\chi^{\phantom{A}}_2=0\\
			\left(\dfrac{2(q-1)\chi^{\phantom{A}}_2+\chi_1^2}{4\kappa+1}\right)^{q/2}|u|^q|v|^qe^{q(1/4-\kappa)t}&\hbox{if }\chi^{\phantom{A}}_2>0.
		\end{array}
		\right.
	\end{eqnarray*}
	\label{lem7}
\end{lemma}

\begin{proof}
	Applying the It\^o formula to $|V^x_t|^2$, we have that
	\begin{eqnarray*}
		\d\,|V^x_t|^2=2\,\langle V^x_t,\,R(\Xi\d B_t,u^x_t)v^x_t\rangle+|R(\Xi\d B_t,u^x_t)v^x_t|^2%\\
		-\left\langle V^x_t,\,R^\sharp_\phi(u^x_t,v^x_t)+\hbox{Ric}^\sharp_\phi(V^x_t)\right\rangle\d t.
	\end{eqnarray*}
	For $q\geqslant2$, write $z_t=\E{|V^x_t|^q}$ and, treating $|V^x_t|^q$ as a function of $|V^x_t|^2$, an application of the It\^o formula to $|V^x_t|^q$ yields
	\begin{eqnarray*}
		\d z_t&=&\frac{q}{2}\E{|V^x_t|^{q-2}\left|R(\Xi\d B_t,u^x_t)v^x_t\right|^2}-\frac{q}{2}\E{|V^x_t|^{q-2}\left\langle V^x_t,\,R^\sharp_\phi(u^x_t,v^x_t)\right\rangle}\d t\\
		&&-\frac{q}{2}\E{|V^x_t|^{q-2}\left\langle V^x_t,\,\hbox{Ric}^\sharp_\phi(V^x_t)\right\rangle}\d t\\
		&&+\frac{1}{2}q(q-2)\E{|V^x_t|^{q-4}\langle V^x_t,\,R(\Xi\d B_t,u^x_t)v^x_t\rangle^2}
	\end{eqnarray*}
	so that, under the given conditions,
	\begin{eqnarray}
		%\begin{array}{rcl}
		\d z_t\leqslant\frac{q}{2}\left\{\chi^{\phantom{A}}_1\E{|V^x_t|^{q-1}|u^x_t|\,|v^x_t|}+(q-1)\,\chi^{\phantom{A}}_2\E{|V^x_t|^{q-2}|u^x_t|^2|v^x_t|^2}\right\}\d t
		-\,q\,\kappa z_t\d t.
		%\end{array}
		\label{eqn11c}
	\end{eqnarray}

	\vskip 6pt
	If $\chi^{\phantom{A}}_2=0$, then $\chi^{\phantom{A}}_1=C_2(\phi)$. we apply the H\"older inequality with conjugate indices $p'=q/(q-1)$ and $q'=q$, as well as Lemma \ref{lem4}, to \eqref{eqn11c} to get
	\begin{eqnarray*}
		\d z_t\leqslant\frac{q}{2}C_2(\phi)\,z_t^{1-1/q}\E{|u^x_t|^q|v^x_t|^q}^{1/q}\d t-q\kappa z_t\d t
		\leqslant\frac{q}{2}C_2(\phi)\,|u|\,|v|\,z_t^{1-1/q}e^{-2\kappa t}\d t-q\kappa z_t\d t. 
	\end{eqnarray*}
	Now, by letting $w=z^{1/q}$, we have %$\d z/(qz)=\d w/w$ and so
	\[\d w_t\leqslant\frac{1}{2}\left\{C_2(\phi)\,|u|\,|v|\,e^{-2\kappa t}-2\,\kappa\,w_t\right\}\d t\]
	%
	%Thus,
	%
	%\[\d\,(e^{\kappa t}w_t)\leqslant\frac{1}{2}C_2(\phi)\,|u|\,|v|\,e^{-\kappa t}\d t\]
	%
	so that, as $w_0=0$, 
	\[w_t\leqslant\frac{1}{2}C_2(\phi)\,|u|\,|v|\,e^{-\kappa t}\int_0^te^{-\kappa s}\d s\leqslant\frac{1}{2\kappa}C_2(\phi)\,|u|\,|v|\,e^{-\kappa t}\]
	as required.
	
	\vskip 6pt
	When $\chi^{\phantom{A}}_2>0$, we need to consider the cases $q=2$ and $q>2$ separately. First, we assume that $q=2$. Then, by H\"older's inequality and Lemma \ref{lem4}, it follows from \eqref{eqn11c} that
	\begin{eqnarray}
		\begin{array}{rcl}
			\d z_t&\leqslant&\chi^{\phantom{A}}_1\,z_t^{1/2}\E{|u^x_t|^4}^{1/4}\E{|v^x_t|^4}^{1/4}\d t\vspace{0.1cm}\\
			&&+\chi^{\phantom{A}}_2\E{|u^x_t|^4}^{1/2}\E{|v^x_t|^4}^{1/2}\d t-2\,\kappa\,z_t\d t\vspace{0.1cm}\\
			%&\leqslant&\chi^{\phantom{A}}_1\,z_t^{1/2}|u|\,|v|\,e^{-2\kappa t}\d t+\chi^{\phantom{A}}_2|u|^2|v|^2e^{-4\kappa t}\d t-2\,\kappa z_t\d t\vspace{0.1cm}\\
			%&\leqslant&\green{\frac{\chi^{\phantom{A}}_1}{2}\left(z_t+|u|^2|v|^2e^{-4\kappa t}\right)\d t+\chi^{\phantom{A}}_2|u|^2|v|^2e^{-4\kappa t}\d t-2\,\kappa z_t\d t}\\
			&\leqslant&\dfrac{1}{2}\left(z_t+\chi_1^2|u|^2|v|^2e^{-4\kappa t}\right)\d t+\chi^{\phantom{A}}_2|u|^2|v|^2e^{-4\kappa t}\d t-2\,\kappa z_t\d t\\
			%&=&\green{\left\{\left(\frac{\chi^{\phantom{A}}_1}{2}-2\,\kappa\right)z_t+\left(\frac{\chi^{\phantom{A}}_1}{2}+\chi^{\phantom{A}}_2\right)|u|^2|v|^2e^{-4\kappa t}\right\}\d t,}\\
			&=&\left\{\left(\dfrac{1}{2}-2\,\kappa\right)z_t+\left(\dfrac{\chi_1^2}{2}+\chi^{\phantom{A}}_2\right)|u|^2|v|^2e^{-4\kappa t}\right\}\d t. 
		\end{array}
		\label{eqn11e}
	\end{eqnarray}
	%
	%where the third inequality follows from the inequality $2ab\leqslant a^2+b^2$. 
	This gives that
	%
	%\begin{eqnarray*}
	%\green{\d\left(e^{(2\kappa-\chi^{\phantom{A}}_1/2) t}z_t\right)\leqslant\left(\frac{\chi^{\phantom{A}}_1}{2}+\chi^{\phantom{A}}_2\right)|u|^2|v|^2e^{-(2\kappa+\chi^{\phantom{A}}_1/2)t}\d t} 
	%\end{eqnarray*}
	%
	%
	\begin{eqnarray*}
		\d\left(e^{(2\kappa-1/2) t}z_t\right)\leqslant\left(\frac{\chi_1^2}{2}+\chi^{\phantom{A}}_2\right)|u|^2|v|^2e^{-(2\kappa+1/2)t}\d t 
	\end{eqnarray*}
	so that
	%
	%\begin{eqnarray*}
	%\green{z_t\leqslant\frac{\chi^{\phantom{A}}_1+2\chi^{\phantom{A}}_2}{4\kappa+\chi^{\phantom{A}}_1}|u|^2|v|^2e^{-(2\kappa-\chi^{\phantom{A}}_1/2) t}} 
	%\end{eqnarray*}
	%
	%
	\begin{eqnarray*}
		z_t\leqslant\frac{\chi_1^2+2\chi^{\phantom{A}}_2}{4\kappa+1}|u|^2|v|^2e^{-(2\kappa-1/2) t} 
	\end{eqnarray*}
	as required.
	
	\vskip 6pt
	If $q>2$, let $(p',q')$ and $(p'',q'')$ be two pairs of conjugate indices such that $p'=q/(q-1)$ and $p''=q/(q-2)$, so that $q'=q$ and $q''=q/2(>1)$. Using these two pairs of conjugate indices, an application of H\"older's inequality to \eqref{eqn11c} gives 
	\begin{eqnarray*}
		\d z_t\leqslant\frac{q}{2}\left\{\chi^{\phantom{A}}_1z_t^{1/p'}\E{|u^x_t|^q|v^x_t|^q}^{1/q'}+(q-1)\chi^{\phantom{A}}_2z_t^{1/p''}\E{|u^x_t|^q|v^x_t|^q}^{1/q''}\right\}\d t%\vspace{0.1cm}\\
		-\,\,q\,\kappa\,z_t\d t.
	\end{eqnarray*}
	Applying the H\"older inequality again, it follows from Lemma \ref{lem4} that 
	\begin{eqnarray*}
		\d z_t\leqslant\frac{q}{2}\,z_t\left\{\chi^{\phantom{A}}_1z_t^{-1/q}|u|\,|v|\,e^{-2\kappa t}+(q-1)\,\chi^{\phantom{A}}_2z_t^{-2/q}|u|^2|v|^2e^{-4\kappa t}-2\kappa\right\}\d t.
	\end{eqnarray*}

	\vskip 6pt
	Letting $w=z^{2/q}$, %we have $2\d z/(qz)=\d w/w$. The 
	the above then gives that
	\begin{eqnarray}
		\begin{array}{rcl}
			\d w_t&
			\leqslant&\left\{\chi^{\phantom{A}}_1w_t^{1/2}|u|\,|v|\,e^{-2\kappa t}+(q-1)\,\chi^{\phantom{A}}_2|u|^2|v|^2e^{-4\kappa t}-2\kappa w_t\right\}\d t\vspace{0.1cm}\\
			%\leqslant&\green{\left\{\left(\chi^{\phantom{A}}_1/2-2\,\kappa\right)w_t+((q-1)\,\chi^{\phantom{A}}_2+\chi^{\phantom{A}}_1/2)\,|u|^2|v|^2e^{-4\kappa t}\right\}\d t,}\\
			&\leqslant&\left\{\left(1/2-2\,\kappa\right)w_t+((q-1)\,\chi^{\phantom{A}}_2+\chi_1^2/2)\,|u|^2|v|^2e^{-4\kappa t}\right\}\d t,
		\end{array}
		\label{eqn11f}
	\end{eqnarray}
	%
	%that is
	%
	%\[\green{\d\,\left(e^{-(\chi^{\phantom{A}}_1/2-2\kappa)t}w_t\right)\leqslant\left((q-1)\,\chi^{\phantom{A}}_2+\chi_1/2\right)\,|u|^2|v|^2e^{-(\chi^{\phantom{A}}_1/2+2\kappa)t}\d t,}\]
	%
	%
	%\[\d\,\left(e^{-(1/2-2\kappa)t}w_t\right)\leqslant\left((q-1)\,\chi^{\phantom{A}}_2+\chi_1^2/2\right)\,|u|^2|v|^2e^{-(1/2+2\kappa)t}\d t,\]
	%
	yielding
	\begin{eqnarray*}
		%&&\green{w_t\leqslant\left((q-1)\,\chi^{\phantom{A}}_2+\chi^{\phantom{A}}_1/2\right)\,|u|^2|v|^2e^{(\chi^{\phantom{A}}_1/2-2\kappa)t}\int_0^te^{-(\chi^{\phantom{A}}_1/2+2\kappa)s}\d s}\\
		w_t\leqslant\left((q-1)\,\chi^{\phantom{A}}_2+\chi_1^2/2\right)|u|^2|v|^2e^{(1/2-2\kappa)t}\int_0^te^{-(1/2+2\kappa)s}\d s
	\end{eqnarray*}
	as $w_0=0$. Subsequently
	\begin{eqnarray*}
		w_t&\leqslant&%\green{\frac{(q-1)\,\chi^{\phantom{A}}_2+\chi^{\phantom{A}}_1/2}{2\kappa+\chi^{\phantom{A}}_1/2}|u|^2|v|^2e^{(\chi^{\phantom{A}}_1/2-2\kappa)t}}\\
		%\frac{(q-1)\,\chi^{\phantom{A}}_2+\chi_1^2/2}{2\kappa+1/2}|u|^2|v|^2e^{(1/2-2\kappa)t}\\
		%&\leqslant&%\green{\frac{2\,(q-1)\,\chi^{\phantom{A}}_2+\chi^{\phantom{A}}_1}{4\kappa+\chi^{\phantom{A}}_1}|u|^2|v|^2e^{(\chi^{\phantom{A}}_1/2-2\kappa) t}}\\
		\frac{2\,(q-1)\,\chi^{\phantom{A}}_2+\chi_1^2}{4\kappa+1}|u|^2|v|^2e^{(1/2-2\kappa)t}
	\end{eqnarray*}
	so that the result follows.
\end{proof}

It is possible, for the case $\chi^{\phantom{A}}_2>0$, to bound $\E{|V^x_t|^q}$ in \ref{lem7} differently. For example, if $q=2$, we can replace the expression on the right hand side of the third inequality in \eqref{eqn11e} by
\[\frac{\chi^{\phantom{A}}_1}{2}\left(z_t+|u|^2|v|^2e^{-4\kappa t}\right)\d t+\chi^{\phantom{A}}_2|u|^2|v|^2e^{-4\kappa t}\d t-2\,\kappa z_t\d t\]
and similarly, if $q>2$, replace the expression on the right hand side of the second inequality in \eqref{eqn11f} by
\[\left\{\left(\chi^{\phantom{A}}_1/2-2\,\kappa\right)w_t+((q-1)\,\chi^{\phantom{A}}_2+\chi^{\phantom{A}}_1/2)\,|u|^2|v|^2e^{-4\kappa t}\right\}\d t.\]
Then, following a similar analysis, we would obtain
\[\E{|V^x_t|^q}\leqslant\left(\frac{2(q-1)\chi^{\phantom{A}}_2+\chi^{\phantom{A}}_1}{4\kappa+\chi_1}\right)^{q/2}|u|^q|v|^qe^{q(\chi^{\phantom{A}}_1/4-\kappa)t}.\]
This feature also appears in the following Lemma. As the results of these two lemmas are used in the proof of Proposition 4, the consequence of this is that we could bound the Lipschitz constant $C_2(f_h)$ in that  proposition differently.

\vskip 6pt
We also need a version of Lemma \ref{lem4a} for $V^x_t$ and $V^y_t$. For this, similarly to the definitions for $v^x_t$ and $v^y_t$, for $x,y\in\M$ and $v,u\in T_x(\M)$, let $V_t^x$ be as given in \eqref{eqn12c} and let $V^y_t$ be the solution to \eqref{eqn12c} with the underlying path $X_{x,t}$ replaced by $Y_{y,t}$ and with $u$ and $v$ there replaced by $\Pi_{\gamma_{x,y}}(u)$ and $\Pi_{\gamma_{x,y}}(v)$ respectively, where $Y_{y,t}$ is constructed as in the proof of Theorem 2 with $\ell=6$ and where, as earlier, $\gamma_{x,y}$ denotes a minimal geodesic from $y$ to $x$ and $\Pi_{\gamma_{x,y}}$ denotes the parallel transport from $T_y(\M)$ to $T_x(\M)$ along $\gamma_{x,y}$. Then, we have the following result, recalling that, if $\chi^{\phantom{A}}_2=0$, then $\chi^{\phantom{A}}_1=L({\rm Ric}^\sharp_\phi)=C_2(\phi)$ and $L(R^\sharp_\phi)=C_3(\phi)$.

\begin{lemma}
	Assume that the conditions of Lemma $\ref{lem7}$ are satisfied. Assume further that ${\rm Ric}^\sharp_\phi$, $R^\sharp_\phi$ and $R$ are all Lipschitz with finite Lipschitz constants $L({\rm Ric}^\sharp_\phi)$, $L(R^\sharp_\phi)$ and $L(R)$ respectively. Then, if $\chi^{\phantom{A}}_2=0$,
	\[\E{\left|\Pi_{\gamma^{\phantom{A}}_{X_{x,t},Y_{y,t}}}(V^x_t)-V^y_t\right|}\leqslant\left(\frac{1}{4\kappa}C_3(\phi)+\frac{3}{4\kappa^2}C_2(\phi)^2\right)\rho(x,y)\,|u|\,|v|\,e^{-\kappa t}\]
	while, if $\chi^{\phantom{A}}_2>0$,
	\begin{eqnarray*}
		\E{\left|\Pi_{\gamma^{\phantom{A}}_{X_{x,t},Y_{y,t}}}(V^x_t)-V^y_t\right|^2}%\\
		\leqslant\left(\frac{\beta_1}{4\kappa+1}+\frac{\beta_2}{3\kappa+1}+\frac{\beta_3}{2\kappa+1}\right)\rho(x,y)^2|u|^2|v|^2e^{(1-2\kappa)t}
	\end{eqnarray*}
	where $\beta_i$, $i=1,2,3$, are as given in Proposition \red{4}.
	\label{lem7a}
\end{lemma}

\begin{proof}
	For simplicity, we write
	\[W_t=\Pi_{\gamma^{\phantom{A}}_{X_{x,t},Y_{y,t}}}(V^x_t)-V^y_t.\]
	Then, an application of It\^o formula to $|W_t|^2$ gives that
	\begin{eqnarray}
		\begin{array}{rcl}
			\d\,|W_t|^2%\vspace{0.1cm}\\
			&=&2\,\langle W_t,\,\Pi_{\gamma^{\phantom{A}}_{X_{x,t},Y_{y,t}}}(R(\Xi(X_{x,t})\d B_t,u^x_t)v^x_t)\rangle\vspace{0.1cm}\\
			&&-2\,\langle W_t,\,R(\Pi_{\gamma^{\phantom{A}}_{X_{x,t},Y_{y,t}}}\Xi(X_{x,t})\d B_t,u^y_t)v^y_t\rangle\\
			&&+\left|\Pi_{\gamma^{\phantom{A}}_{X_{x,t},Y_{y,t}}}\!\!\!(R(\Xi(X_{x,t})\d B_t,u^x_t)v^x_t)-R(\Pi_{\gamma^{\phantom{A}}_{X_{x,t},Y_{y,t}}}\!\!\!\Xi(X_{x,t})\d B_t,u^y_t)v^y_t\right|^2\vspace{0.1cm}\\
			&&-\langle W_t,\,\Pi_{\gamma^{\phantom{A}}_{X_{x,t},Y_{y,t}}}(R^\sharp_\phi(X_{x,t})(u^x_t,v^x_t))-R^\sharp_\phi(Y_{y,t})(u^y_t,v^y_t)\rangle\d t\vspace{0.1cm}\\
			&&-\langle W_t,\,\Pi_{\gamma^{\phantom{A}}_{X_{x,t},Y_{y,t}}}(\hbox{Ric}^\sharp_\phi(X_{x,t})(V^x_t))-\hbox{Ric}^\sharp_\phi(Y_{y,t})(V^y_t)\rangle\d t.
		\end{array}
		\label{eqn11b}
	\end{eqnarray}

	Assume first that $\chi^{\phantom{A}}_2\not=0$. As in the proof of Lemma \ref{lem4a}, for the third term on the right hand side of \eqref{eqn11b}, we have
	\begin{eqnarray*}
		&&|\Pi_{\gamma^{\phantom{A}}_{X_{x,t},Y_{y,t}}}(R(\Xi(X_{x,t})\d B_t,u^x_t)v^x_t)-R(\Pi_{\gamma^{\phantom{A}}_{X_{x,t},Y_{y,t}}}\Xi(X_{x,t})\d B_t,u^y_t)v^y_t|^2\\
		%&\leqslant&2\,|R(\Pi_{\gamma^{\phantom{A}}_{X_{x,t},Y_{y,t}}}\Xi(X_{x,t})\d B_t,\tilde u^x_t)\tilde v^x_t-R(\Pi_{\gamma^{\phantom{A}}_{X_{x,t},Y_{y,t}}}\Xi(X_{x,t})\d B_t,u^y_t)v^y_t|^2\\
		%&&+2\,|\Pi_{\gamma^{\phantom{A}}_{X_{x,t},Y_{y,t}}}(R(\Xi(X_{x,t})\d B_t,u^x_t)v^x_t)-R(\Pi_{\gamma^{\phantom{A}}_{X_{x,t},Y_{y,t}}}\Xi(X_{x,t})\d B_t,\tilde u^x_t)\tilde v^x_t|^2\\
		&\leqslant&4\,|R(\Pi_{\gamma^{\phantom{A}}_{X_{x,t},Y_{y,t}}}\Xi(X_{x,t})\d B_t,\tilde u^x_t-u^y_t)\tilde v^x_t|^2%\\
		+4\,|R(\Pi_{\gamma^{\phantom{A}}_{X_{x,t},Y_{y,t}}}\Xi(X_{x,t})\d B_t,u^y_t)(\tilde v^x_t-v^y_t)|^2\\
		&&+2\,|R(\Xi(X_{x,t})\d B_t,u^x_t)v^x_t)-\Pi_{\gamma^{\phantom{A}}_{X_{x,t},Y_{y,t}}}^{-1}(R(Y^y_t))(\Xi(X_{x,t})\d B_t,u^x_t)v^x_t|^2\d t\\
		&\leqslant&4\,\chi^{\phantom{A}}_2\left(|\tilde u^x_t-u^y_t|^2|\tilde v^x_t|^2+|u^y_t|^2|\tilde v^x_t-v^y_t|^2\right)\d t
		+2\,m\,L(R)^2\rho(X_{x,t},Y_{y,t})^2|u^x_t|^2|v^x_t|^2\d t.
	\end{eqnarray*} 
	Thus, by H\"older's inequality, Theorem 2, as well as Lemmas \ref{lem4} and \ref{lem4a},
	\begin{eqnarray*}
		&&\E{|\Pi_{\gamma^{\phantom{A}}_{X_{x,t},Y_{y,t}}}(R(\Xi(X_{x,t})\d B_t,u^x_t)v^x_t)-R(\Pi_{\gamma^{\phantom{A}}_{X_{x,t},Y_{y,t}}}\Xi(X_{x,t})\d B_t,u^y_t)v^y_t|^2}\\
		%&\leqslant&4\,\chi^{\phantom{A}}_2\E{|\tilde u^x_t-u^y_t|^4}^{1/2}\E{|\tilde v^x_t|^4}^{1/2}\d t\\
		%&&+4\,\chi^{\phantom{A}}_2\E{|u^y_t|^4}^{1/2}\E{|\tilde v^x_t-v^y_t|^4}^{1/2}\d t\\
		%&&+2\,m\,L(R)^2\E{\rho(X_{x,t},Y_{y,t})^6}^{1/3}\E{|u^x_t|^6}^{1/3}\E{|v^x_t|^6}^{1/3}\d t\\
		&\leqslant&\rho(x,y)^2|u|^2|v|^2\left\{8\,\chi^{\phantom{A}}_2\left(\frac{L({\rm Ric}^\sharp_\phi)}{2\kappa}\right)^2e^{-4\kappa t}+2\,m\,L(R)^2e^{-6\kappa t}\right\}\d t.
	\end{eqnarray*} 

	Similarly, for the fourth term on the right hand side of \eqref{eqn11b}, 
	\begin{eqnarray}
		\begin{array}{rcl}
			&&-\langle W_t,\,\Pi_{\gamma^{\phantom{A}}_{X_{x,t},Y_{y,t}}}(R^\sharp_\phi(X_{x,t})(u^x_t,v^x_t))-R^\sharp_\phi(Y_{y,t})(u^y_t,v^y_t)\rangle\\
			%&=&-\langle W_t,\,\Pi_{\gamma^{\phantom{A}}_{X_{x,t},Y_{y,t}}}(R^\sharp_\phi(X_{x,t})(u^x_t,v^x_t))-R^\sharp_\phi(Y_{y,t})(\tilde u^x_t,\tilde v^x_t)\rangle\\
			%&&-\langle W_t,\,R^\sharp_\phi(Y_{y,t})(\tilde u^x_t,\tilde v^x_t))-R^\sharp_\phi(Y_{y,t})(u^y_t,v^y_t)\rangle\vspace{0.1cm}\\
			&\leqslant&|W_t|\,|R^\sharp_\phi(X_{x,t})(u^x_t,v^x_t)-\Pi_{\gamma^{\phantom{A}}_{X_{x,t},Y_{y,t}}}^{-1}(R^\sharp_\phi(Y_{y,t}))(u^x_t,v^x_t)|\\
			&&+|W_t|\left\{|R^\sharp_\phi(Y_{y,t})((\tilde u^x_t-u^y_t,\tilde v^x_t)|+|R^\sharp_\phi(Y_{y,t})(u^y_t,\tilde v^x_t-v^y_t))|\right\}\vspace{0.1cm}\\
			&\leqslant&L(R^\sharp_\phi)\,|W_t|\,\rho(X_{x,t},Y_{y,t})\,|u^x_t|\,|v^x_t|%\vspace{0.1cm}\\
			+\chi^{\phantom{A}}_1|W_t|\left\{|\tilde u^x_t-u^y_t|\,|\tilde v^x_t|+|u^y_t|\,|\tilde v^x_t-v^y_t|\right\}.
		\end{array}
		\label{eqn11g}
	\end{eqnarray}
	Again, by H\"older's inequality, Theorem 2, as well as Lemmas \ref{lem4} and \ref{lem4a}, we have
	\begin{eqnarray*}
		&&-\E{\langle W_t,\,\Pi_{\gamma^{\phantom{A}}_{X_{x,t},Y_{y,t}}}(R^\sharp_\phi(X_{x,t})(u^x_t,v^x_t))-R^\sharp_\phi(Y_{y,t})(u^y_t,v^y_t)\rangle}\\
		&\leqslant&L(R^\sharp_\phi)\E{|W_t|^2}^{1/2}\E{\rho(X_{x,t},Y_{y,t})^6}^{1/6}\E{|u^x_t|^6}^{1/6}\E{|v^x_t|^6}^{1/6}\\
		&&+\chi^{\phantom{A}}_1\E{|W_t|^2}^{1/2}\E{|\tilde u^x_t-u^y_t|^4}^{1/4}\E{|\tilde v^x_t|^4}^{1/4}\\
		&&+\chi^{\phantom{A}}_1\E{|W_t|^2}^{1/2}\E{|u^y_t|^4}^{1/4}\E{|\tilde v^x_t-v^y_t|^4}^{1/4}\\
		&\leqslant&\rho(x,y)\,|u|\,|v|\,\E{|W_t|^2}^{1/2}\left\{L(R^\sharp_\phi)\,e^{-3\kappa t}+\frac{\chi^{\phantom{A}}_1L({\rm Ric}^\sharp_\phi)}{\kappa}e^{-2\kappa t}\right\}\\
		&\leqslant&\frac{1}{2}\E{|W_t|^2}+\frac{1}{2}\rho(x,y)^2|u|^2|v|^2\left(L(R^\sharp_\phi)\,e^{-3\kappa t}+\frac{\chi^{\phantom{A}}_1}{\kappa}L({\rm Ric}^\sharp_\phi)\,e^{-2\kappa t}\right)^2.
	\end{eqnarray*}

	Finally, for the fifth term of the right hand side of \eqref{eqn11b},
	\begin{eqnarray}
		\begin{array}{rl}
			&\quad-\,\langle W_t,\,\Pi_{\gamma^{\phantom{A}}_{X_{x,t},Y_{y,t}}}(\hbox{Ric}^\sharp_\phi(X_{x,t})(V^x_t))-\hbox{Ric}^\sharp_\phi(Y_{y,t})(V^y_t)\rangle\vspace{0.1cm}\\
			%&=&-\left\langle W_t,\,\text{Ric}^\sharp_\phi(Y_{y,t})(\tilde V^x_t)-\text{Ric}^\sharp_\phi(Y_{y,t})(V^y_t)\right\rangle\vspace{0.1cm}\\ 
			%&&-\left\langle W_t,\,\Pi_{\gamma^{\phantom{A}}_{X_{x,t},Y_{y,t}}}(\text{Ric}^\sharp_\phi(X_{x,t})(V^x_t))-\text{Ric}^\sharp_\phi(Y_{y,t})(\tilde V^x_t)\right\rangle\vspace{0.1cm}\\
			&\leqslant-\left(\text{Ric}+\text{Hess}^\phi\right)(W_t,\,W_t)%\vspace{0.1cm}\\
			+\,|W_t|\,|\text{Ric}^\sharp_\phi(X_{x,t})(V^x_t)-\Pi^{-1}_{\gamma^{\phantom{A}}_{X_{x,t},Y_{y,t}}}(\text{Ric}^\sharp_\phi(Y_{y,t}))(V^x_t)|\vspace{0.1cm}\\
			&\leqslant-\,2\kappa\,|W_t|^2+L(\text{Ric}^\sharp_\phi)\,|W_t|\,|V^x_t|\,\rho(X_{x,t},\,Y_{y,t}).
		\end{array}
		\label{eqn11d}
	\end{eqnarray} 
	With a similar argument to that above, we obtain that
	\begin{eqnarray*}
		&&-\E{\langle W_t,\,\Pi_{\gamma^{\phantom{A}}_{X_{x,t},Y_{y,t}}}(\hbox{Ric}^\sharp_\phi(X_{x,t})(V^x_t))-\hbox{Ric}^\sharp_\phi(Y_{y,t})(V^y_t)\rangle}\\
		&\leqslant&-2\kappa\,\E{|W_t|^2}+L(\text{Ric}^\sharp_\phi)\,\E{|W_t|^2}^{1/2}\E{|V^x_t|^4}^{1/4}\E{\rho(X_{x,t},\,Y_{y,t})^4}^{1/4}\\
		%&\leqslant&-2\kappa\,\E{|W_t|^2}\\
		%&&+L(\text{Ric}^\sharp_\phi)\left(\frac{6\chi^{\phantom{A}}_2+\chi_1^2}{4\kappa+1}\right)^{1/2}\E{|W_t|^2}^{1/2}\rho(x,y)\,|u|\,|v|\,e^{(1/4-2\kappa)t}\\
		&\leqslant&\left\{-2\kappa+\frac{1}{2}\right\}\E{|W_t|^2}+\frac{1}{2}\frac{6\chi^{\phantom{A}}_2+\chi_1^2}{4\kappa+1}L(\text{Ric}^\sharp_\phi)^2\rho(x,y)^2|u|^2|v|^2e^{(1/2-4\kappa)t}
	\end{eqnarray*} 
	as $\chi^{\phantom{A}}_2\not=0$. 
	
	\vskip 6pt
	Returning to \eqref{eqn11b}, the above calculations imply that
	\begin{eqnarray*}
		&&\d\E{|W_t|^2}\leqslant\alpha\E{|W_t|^2}\d t\\
		&&\qquad+\rho(x,y)^2|u|^2|v|^2\left(\beta^*_1e^{-6\kappa t}+\beta_2e^{-5\kappa t}+\beta_3e^{-4\kappa t}+\beta_4^*e^{(1/2-4\kappa)t}\right)\d t,
	\end{eqnarray*}
	where $\alpha=1-2\kappa$, $\beta^*_1=2mL(R)^2+L(R^\sharp_\phi)^2/2$, $\beta^*_4=(\beta_1-\beta_1^*)/2$ and $\beta_1$, $\beta_2$, $\beta_3$ are as given in the statement of the Lemma. Hence,
	\begin{eqnarray*}
		\d\left(e^{-\alpha t}\E{|W_t|^2}\right)
		\leqslant\rho(x,y)^2|u|^2|v|^2\left(\beta^*_1e^{-6\kappa t}+\beta_2e^{-5\kappa t}+\beta_3e^{-4\kappa t}+\beta^*_4e^{(1/2-4\kappa)t}\right)e^{-\alpha t}\d t,
	\end{eqnarray*}
	so that
	\begin{eqnarray*}
		\E{|W_t|^2}
		&\leqslant&%\rho(x,y)^2|u|^2|v|^2\left(\frac{\beta^*_1}{6\kappa+\alpha}+\frac{\beta_2}{5\kappa+\alpha}+\frac{\beta_3}{4\kappa+\alpha}+\frac{\beta^*_4}{4\kappa+\alpha-1/2}\right)e^{\alpha t}\\
		%&=&
		\rho(x,y)^2|u|^2|v|^2\left(\frac{\beta_1}{4\kappa+1}+\frac{\beta_2}{3\kappa+1}+\frac{\beta_3}{2\kappa+1}\right)e^{\alpha t}
	\end{eqnarray*}
	as required.
	
	\vskip 6pt
	If $\chi^{\phantom{A}}_2=0$, then \eqref{eqn11b} gives %reduces to 
	%
	%\begin{eqnarray*}
	%\d\,|W_t|^2&=&-\langle W_t,\,\Pi_{\gamma^{\phantom{A}}_{X_{x,t},Y_{y,t}}}(R^\sharp_\phi(X_{x,t})(u^x_t,v^x_t))-R^\sharp_\phi(Y_{y,t})(u^y_t,v^y_t)\rangle\d t\\
	%&&-\langle W_t,\,\Pi_{\gamma^{\phantom{A}}_{X_{x,t},Y_{y,t}}}(\hbox{Ric}^\sharp_\phi(X_{x,t})(V^x_t))-\hbox{Ric}^\sharp_\phi(Y_{y,t})(V^y_t)\rangle\d t,
	%\end{eqnarray*} 
	%
	%and so consequently
	%
	\begin{eqnarray*}
		\d\,|W_t|&=&-\frac{1}{2|W_t|}\langle W_t,\,\Pi_{\gamma^{\phantom{A}}_{X_{x,t},Y_{y,t}}}(R^\sharp_\phi(X_{x,t})(u^x_t,v^x_t))-R^\sharp_\phi(Y_{y,t})(u^y_t,v^y_t)\rangle\d t\\
		&&-\frac{1}{2|W_t|}\langle W_t,\,\Pi_{\gamma^{\phantom{A}}_{X_{x,t},Y_{y,t}}}(\hbox{Ric}^\sharp_\phi(X_{x,t})(V^x_t))-\hbox{Ric}^\sharp_\phi(Y_{y,t})(V^y_t)\rangle\d t.
	\end{eqnarray*} 
	Thus, from the above analysis \eqref{eqn11g} and \eqref{eqn11d} of the third and fourth terms respectively on the right hand side of \eqref{eqn11b}, it follows that
	\begin{eqnarray*}
		\d\,|W_t|
		&\leqslant&\frac{1}{2}\left\{L(R^\sharp_\phi)\,\rho(X_{x,t},Y_{y,t})\,|u^x_t|\,|v^x_t|+\chi^{\phantom{A}}_1\left(|\tilde u^x_t-u^y_t|\,|\tilde v^x_t|+|u^y_t|\,|\tilde v^x_t-v^y_t|\right)\right\}\d t\\
		&&+\left\{-\kappa\,|W_t|+\frac{1}{2}L(\hbox{Ric}^\sharp_\phi)\,\rho(X_{x,t},\,Y_{y,t})\,|V^x_t|\right\}\d t.
	\end{eqnarray*}
	Using the fact that, if $\chi^{\phantom{A}}_2=0$, then $\chi^{\phantom{A}}_1=L({\rm Ric}^\sharp_\phi)=C_2(\phi)$ and $L(R^\sharp_\phi)=C_3(\phi)$, we then have
	\begin{eqnarray*}
		\d\left(e^{\kappa t}\E{|W_t|}\right)&\leqslant&\frac{1}{2}\rho(x,y)\,|u|\,|v|\,\left\{C_3(\phi)\,e^{-2\kappa t}+\frac{1}{\kappa}C_2(\phi)^2e^{-\kappa t}\right\}\d t\\
		&&+\frac{1}{4\kappa}\,C_2(\phi)^2\rho(x,y)\,|u|\,|v|\,e^{-\kappa t}\d t
	\end{eqnarray*}
	yielding
	\begin{eqnarray*}
		\E{|W_t|}
		\leqslant\left\{\frac{1}{4\kappa}C_3(\phi)+\frac{3}{4\kappa^2}C_2(\phi)^2\right\}\rho(x,y)\,|u|\,|v|\,e^{-\kappa t}
	\end{eqnarray*}
	as required.
\end{proof}	
\section*{Acknowledgements}
We are grateful to Dennis Barden for permission to include his unpublished work (Lemma 3 in Appendix A of Supplementary Material), jointly with HL. This result lays the geometric foundation for our analysis of the distance between the pair of diffusions. KB acknowledges partial support from grants EPSRC EP/V048104/1, NSF 2015374 and NIH R37-CA214955. 
%%%%%%%%%%%%%%%%%%%%%%%%%%%%%%%%%%%%%%%%%%
\bibliography{biblio}
\bibliographystyle{plainnat}

\end{document}